\documentclass[12pt,fleqn,a4paper]{article}

\usepackage{epsfig,amssymb,amsthm,amsmath,mathrsfs,latexsym}

\newcommand{\rem}[1]{} % fu"hrt zum `Nulloutput'

\newcommand{\R}{{\mathbb{R}}}
\newcommand{\C}{{\mathbb{C}}}
\newcommand{\Z}{{\mathbb{Z}}}

\newcommand{\T}{{\mathbb{T}}}

\newcommand{\EM}{{\cal E \!\:\!\! M}}
\newcommand{\EC}{{\cal E \!\:\! C}}
\newcommand{\IC}{{\cal I \!\:\! C}}
\newcommand{\res}{\mathop{\mathrm{Res}}}
\newcommand{\rank}{\mathop{\mathrm{rank}}}
\newcommand{\const}{{\rm const}}

\newcommand{\J}{{\mathbf{J}}}
\newcommand{\trace}{\mathop{\mathrm{tr}}}
\newcommand{\Ad}{{\mathrm{Ad}}}
\newcommand{\xx}{{\mathbf x}}
\newcommand{\yy}{{\mathbf y}}
 
\newcommand{\polpo}{ {\cal P}}
\newcommand{\polw}{ {\cal O}}

\newcommand{\xib}{{\boldsymbol{\bf\xi}}}
\newcommand{\etab}{{\boldsymbol{\bf\eta}}}
 
\newcommand{\Pro}{\mathcal{P}} 

\newcommand{\scap}[2]{\langle#1; #2\rangle}
\newcommand{\e}{ {\bf e}}

\newcommand{\oo}{{\mathfrak{so}}}
\renewcommand{\gg}{{\mathfrak{g}}}

\newcommand{\id}{{\rm id}}
\newcommand{\dee}{{\rm d}}

\newcommand{\X}[1]{{X_{\textstyle \! #1}}}

\newtheorem{theorem}{Theorem}[section]
\newtheorem{lemma}[theorem]{Lemma}
\newtheorem{proposition}[theorem]{Proposition}
\newtheorem{corollary}[theorem]{Corollary}
\newtheorem{definition}[theorem]{Definition}
\newtheorem{rmrk}[theorem]{Remark}

\newenvironment{remark}{\begin{rmrk} \begin{rm}}{\end{rm} \end{rmrk}}

\begin{document}

\title{The degenerate C.~Neumann system~I~: \\
       symmetry reduction and convexity}

\author
{
{\protect\normalsize Holger R.\ Dullin} \\
{\protect\footnotesize\protect\it
School of Mathematics and Statistics,  University of Sydney} \\[-2mm]
{\protect\footnotesize\protect\it
Sydney NSW 2006, Australia\footnote{ email:
Holger.Dullin@sydney.edu.au
}.} \\
\\
{\protect\normalsize Heinz Han{\ss}mann} \\
{\protect\footnotesize\protect\it
Mathematisch Instituut, Universiteit Utrecht} \\[-2mm]
{\protect\footnotesize\protect\it
3508 TA Utrecht, The Netherlands.}
}

\date{\protect\normalsize 21 October 2011}

\maketitle

\begin{abstract}
\noindent
The C. Neumann system describes a particle on the sphere $S^n$ under the
influence of a potential that is a quadratic form.
We study the case that the quadratic form has $\ell + 1$ distinct eigenvalues
with multiplicity.
Each group of $m_{\sigma}$ equal eigenvalues gives rise to an
$O(m_{\sigma})$--symmetry in configuration space.
The combined symmetry group $G$ is a direct product of $\ell+1$ such factors, 
and its cotangent lift has an $\Ad^*$--equivariant Momentum mapping.
Regular reduction leads to the Rosochatius system on $S^\ell$, which has the
same form as the Neumann system albeit for an additional effective potential.

To understand how the reduced systems fit together we use singular
reduction to construct an embedding of the reduced Poisson space
$T^*S^n/G$ into $\R^{3\ell+3}$.
The global geometry is described, in particular the bundle structure
that appears as a result of the superintegrability of the system.
We show how the reduced Neumann system separates in elliptical-spherical
co-ordinates.
We derive the action variables and frequencies as complete hyperelliptic
integrals of genus $\ell$.
Finally we prove a convexity result for the image of the Casimir mapping
restricted to the energy surface.
\end{abstract}

\section{Introduction}

The C.~Neumann system is one of the few examples of an integrable system
with $n$~degrees of freedom, where $n$ is an arbitrary positive integer.
It describes a particle moving on the sphere $S^n$ under the influence of
a linear force. 
As C.G.~Jacobi's student, Carl Neumann studied the case $n=2$ in his
thesis~\cite{neumann56}.
The general case has been beautifully described by
Moser~\cite{Moser80,Moser80b,Moser83}.
Kn\"orrer~\cite{Knoerrer82} elucidated the connection to the geodesic flow
on ellipsoids, see also~\cite{Veselov94}.
Later Kn\"orrer~\cite{Knoerrer85} also showed that near hyperbolic critical
values of the Integral mapping the Kolmogorov non-degeneracy condition
holds so that KAM theory can be applied.
Whether this is true for all regular values is still open. 
Devaney~\cite{Devaney78} noticed that the system has transversal homoclinic
intersections while being Liouville integrable.
Ratiu~\cite{Ratiu81} realised that the system can be written as dynamics
on the adjoint orbit of $SO(n) \ltimes Sym(n)$ on the Lie algebra of this
semi-direct product.
The action variables and their Picard--Fuchs equations haven been derived
in~\cite{dullin01}.

All these works assume the generic case that the quadratic potential
$\frac{1}{2} \scap{\xx}{{\bf A}\xx}$, $\xx \in \R^{n+1}$ of the system
embedded in $\R^{n+1}$ has $n+1$ {\em distinct} eigenvalues
$a_0 < a_1 < \ldots < a_n$. 
The present paper studies the degenerate case in which the eigenvalues $a_\nu$
are not all distinct.
This modification leads to the Rosochatius system, an
integrable $(n+1)$--parameter deformation of the Neumann system, which we
derive and analyse.
The fact that the degenerate Neumann system is also integrable has been
shown in~\cite{Zhangju92}.
This is non-trivial because the smooth integrals
(\eqref{eqn:genIntegrals}, see below) 
of the non-degenerate system found by Uhlenbeck~\cite{Moser80} become
singular in this limit.
So even though equal eigenvalues make the system simpler in the sense that
global symmetry is introduced, a non-trivial transition takes place.

The non-degenerate Neumann systems on $S^n$ with $n > 1$ does not admit
{\em globally defined} continuous symmetries.
The symmetry group of the generic Neumann system is the direct product
$\Z_2^{n+1}$ of $n+1$ factors $O(1) \cong \Z_2$ and hence discrete.
Correspondingly, there are no global smooth actions, i.e.\ integrals of
motion that globally generate periodic flows.
By the Liouville--Arnold theorem almost everywhere there exist local actions,
but usually they cannot be extended to smooth global integrals.
Hence locally there is a free $\T^n$--action (defined by the commuting flows
of local actions $I_j$, cf.~\cite{dullin01}), but globally there is not even a
$\T^1$--action.
Nevertheless, the system can be integrated by separation of variables.
Explicit solutions can be derived  in terms of $\theta$--functions of
genus $n$, see~\cite{Moser80b}.

The {\em degenerate} Neumann system on $S^n$ admits a large global symmetry group,
and hence possesses corresponding global actions.
For each group of $m_\sigma$ equal eigenvalues we have an
$O(m_\sigma)$ symmetry that acts by rotation of the corresponding group of
co-ordinates and by cotangent lift on the momenta. 
In general, the potential has $\ell + 1$ distinct eigenvalues
$b_0 < \ldots < b_\ell$ of multiplicities $m_0, \ldots, m_\ell$.
The corresponding symmetry group of the Neumann systems on~$S^n$ 
is $G = O(m_0) \times O(m_1) \times \ldots \times O(m_{\ell})$.
In the generic case $\ell = n$ that all eigenvalues are distinct
we recover $G = \Z_2^{n+1}$.
The symmetry group $G$ describes the freedom in the choice of the (non-unique)
co-ordinate system in which $\bf A$ is diagonal.

In most of the paper we consider the orthogonal groups as the relevant factors
in the symmetry group~$G$, because this allows a uniform treatment of the
factors~$O(m)$ no matter what $m$ is.
However, sometimes it is more appropriate to restrict in one or several factors
to the special orthogonal group.
This is the case when the singularities of the reduced phase space can be
avoided, and it occurs where a factor of $G$ is~$O(2)$. 
The reason for this is a subtle difference between the invariants of the $O(2)$
and the $SO(2)$ actions, which is not present when the factor is $O(m)$ with
$m>2$.
This difference becomes particularly important when studying Hamiltonian
monodromy, which appears e.g.\ when there are two eigenvalues of multiplicity
two each.
The corresponding situation in the geodesic flow on the ellipsoid has been
studied in \cite{Davidson07a,Davidson07b}.

Our main result is the reduction of the degenerate Neumann system with
the symmetry group~$G$ and a detailed description of the relation of the
reduced system with $\ell$ degrees of freedom with the Rosochatius system.
After describing the setting in the next section, and a review of 
orthogonal group actions in Section~\ref{orthogonalgroupactions},
this reduction is performed in Section~\ref{reduce}.
When the momentum is weakly regular, Marsden--Weinstein
reduction~\cite{marsden74} leads to a system on~$S^\ell$
with an additional effective potential.
Since not all momentum values are weakly regular we perform singular
reduction, which for singular values leads to reduced phase spaces
with singularities and simultaneously embeds all reduced phase spaces
coming from (weakly) regular values as symplectic leaves of a Poisson
structure in $R^{3\ell + 3}$. 
When $m_{\sigma} \geq 3$ for at least one index~$\sigma$ the system is
superintegrable, with a non-commutative symmetry group, and the lower
dimensional tori are generically parametrised by a product of oriented
Gra{\ss}mannian manifolds $G_{m_{\sigma}, 2}$.
Furthermore we compute the relative equilibria and characterise the
Energy--Casimir mapping.
In Section~\ref{rosochatius} we show how the Rosochatius system obtained
by reduction of the degenerate Neumann system can be separated in
elliptical-spherical co-ordinates.
From this the non-trivial actions and frequencies are derived.
In the final Section~\ref{globalstructureoftheflow} we show that the
set of critical values of the Integral mapping of the reduced system
(consisting of the $\ell$~independent integrals of motion that are not
obtained from the symmetry group~$G$) is topologically a $2^{\ell}$--tant
when all Casimirs are nonzero.
Finally, we prove that the image of the Casimir mapping restricted to the
energy surface is a convex set, which in the limit of large energy tends to
a convex polyhedron.

\section{The Neumann System}
\label{theneumannsystem}

The Neumann system describes a particle moving on a sphere under the
influence of a quadratic potential.
We use $\xx = (x_0,x_1,\dots,x_n)$ as co-ordinates in $\R^{n+1}$ in
which the sphere $S^n$ is embedded as $C_1 = \scap{\xx}{\xx} = 1$
with the standard Euclidean scalar product.
The kinetic energy is $T(\dot \xx) = \frac{1}{2} \scap{\dot \xx}{\dot \xx}$
and the potential is $V(\xx) = \frac{1}{2} \scap{\xx}{{\bf A}\xx}$.
Newton's equations are
\begin{equation}  \label{eqn:Newton}
   \ddot \xx = -\nabla V + \lambda \xx
\end{equation}
where the Lagrange multiplier $\lambda$ gives the strength
$\lambda = 2V - 2T$ of the constraining normal force.
The initial conditions must be chosen such that
$\dot C_1 = 2 \scap{\xx}{\dot \xx} = 0$.
To preserve $C_1$ we furthermore require $\ddot C_1 = 0$, which
determines $\lambda$.
By a rotation of the co-ordinate system we can always achieve that ${\bf A}$
is diagonal with diagonal entries $a_{\nu}$, $\nu = 0,1,\dots,n$.
The classical Neumann system has $n=2$, and distinct eigenvalues of 
${\bf A}$ ordered as $a_0 < a_1 < a_2$.
We call a Neumann system (with arbitrary $n$) degenerate if ${\bf A}$
has multiple eigenvalues.

\subsection{Global Hamiltonian description}
\label{globalhamiltoniandescription}

The canonical momenta are $\yy = (y_0, y_1, \ldots y_n)$.
They satisfy the canonical Poisson bracket relations
$[x_{\nu}, y_{\kappa}] = \delta_{\nu \kappa}$,
$[x_{\nu}, x_{\kappa}] =[y_{\nu}, y_{\kappa}] = 0$.
Besides $C_1 = 1$ the momentum vectors are constrained to be tangent to
the sphere, $C_2 = \scap{\xx}{\yy} = 0$.
This embeds the phase space $T^*S^n$ into~$\R^{2n+2}$.
Since we are dealing with a constrained system for which $\xx$ are
not generalized co-ordinates the Lagrangian is degenerate and the
Legendre transformation to the Hamiltonian does not work.
Instead we just write down the Hamiltonian in canonical variables of the
embedding space
\begin{equation}  \label{eqn:Hamiltonian}
   H(\xx, \yy) =  \frac{1}{2} \scap{\yy}{\yy} +  V(\xx)  \, .
\end{equation}
To ensure that the constraints $C_1 = 1$ and $C_2 = 0$ are respected by
the corresponding Hamiltonian vector field, we modify the canonical Poisson
bracket $[ .. \, , .. ]$ to the Dirac Poisson bracket $\{ .. \, , .. \}$,
cf.\ e.g.~\cite{cushman97}.
We then confirm directly that $\{ .. \, , .. \}$ together
with~\eqref{eqn:Hamiltonian} lead to Newton's equations as well,
see~\eqref{eqn:HamEqs} below.

For the Dirac bracket one first calculates the bracket of the constraints.
The only nonzero bracket in our case is  $[C_1, C_2] = 2 C_1$.
Alltogether they form the matrix $\gamma_{ij} = [C_i, C_j]$. 
The original bracket is modified by $[f, C_i] [C_j, g] (\gamma^{-1})_{ij}$. 
This gives
\begin{equation}  \label{eqn:genDirac}
   \{ f, g \}   =   [f,g]  +  \frac{1}{2 C_1} [f, C_1] [C_2, g]
   -  \frac{1}{2 C_1} [f, C_2] [C_1, g]  \, .
\end{equation}
Sometimes we need this bracket in spaces of different dimensions,
in which case we indicate it by a subscript, whence the above bracket
reads $\{ .. \, , .. \}_{2n+2}$.
This new Poisson structure $\{ .. \, , .. \}$ has $C_1$ and $C_2$ as
Casimirs and is explicitly given by 
\begin{equation}  \label{eqn:Dirac}
   \{ x_{\nu}, x_{\kappa} \} = 0 \, ,
   \qquad
   \{ x_{\nu}, y_{\kappa} \} =
   \delta_{\nu \kappa} - \frac{x_{\nu} x_{\kappa}}{C_1} \, ,
   \qquad
   \{ y_{\nu}, y_{\kappa} \} = - \frac{L_{\nu \kappa}}{C_1} \, .
\end{equation}
Here $L_{\nu \kappa} =  x_{\nu} y_{\kappa} - x_{\kappa} y_{\nu}$ are
the components of the angular momentum 
$
    \J : T^*\R^{n+1} \longrightarrow \oo(n+1)^* \simeq \R^{n(n+1)/2}
$.
The Poisson bracket $\{  L_{\nu \kappa}, L_{\mu \lambda} \}$ 
is nonzero if and only if exactly one of the indices $\nu$, $\kappa$
coincides with one of the indices $\mu$, $\lambda$. 
Using $L_{\nu \kappa} = -L_{\kappa \nu}$ the single coinciding index can
be moved so that we obtain all nonzero bracket relations from
\begin{displaymath}
   \{  L_{\nu \kappa}, L_{\kappa \lambda} \} = L_{ \lambda \nu} \,.
\end{displaymath}
Note that
$\{  L_{\nu \kappa}, L_{\kappa \lambda} \}
 = [ L_{\nu \kappa}, L_{\kappa \lambda} ]$
because $[L_{ij} , C_1] = 0$.
Now Hamilton's equations $\dot{f} = \{ f, H \}$ restricted to $T^*S^n$ are
\begin{equation}  \label{eqn:HamEqs}
   \dot{\xx} = \yy, \quad
   \dot{\yy} = -\nabla V + ( \scap{\xx}{\nabla V}-2T) \xx \, .
\end{equation}
Here we used the Casimir $C_2 = 0$ and the identity
$\sum L_{\nu \kappa} y_{\kappa} = x_{\nu} \sum y_{\kappa}^2 $.
These equations are equivalent to
Newton's equations~\eqref{eqn:Newton}.

The Neumann system with distinct eigenvalues $a_{\nu}$ is Liouville integrable.
The polynomial integrals $\tilde F_0, \dots, \tilde F_n$ given
in~\cite{Moser80} are
\begin{equation}  \label{eqn:genIntegrals}
    \tilde F_{\nu}(\xx,\yy) = x_{\nu}^2 +
    \sum_{\mu \neq \nu}^n \frac{L_{\nu\mu}^2}{a_{\nu} - a_{\mu}}, 
\end{equation}
and they are independent up to the relation $\sum \tilde F_{\nu} = C_1$.
When restricting to $T^*S^n$ the Hamiltonian can be expressed as
\begin{displaymath}
   H = \frac{1}{2} \sum_{\nu = 0}^n a_{\nu} \tilde F_{\nu} \,.
\end{displaymath}
Surprisingly, these integrals were not classically known~\cite{Moser80}.

Equations \eqref{eqn:HamEqs} are valid for dynamics on $S^n$ with
arbitrary potential.
Integrable systems in this class of Hamiltonians include not only
the Neumann system with $V(\xx) =  \frac{1}{2} \sum a_{\nu} x_{\nu}^2$
but also the spherical pendulum with $V(\xx) = x_0$ or the Rosochatius
system~\cite{rosochatius77, Moser80b, Macfarlane92} with
$V(\xx) =  \frac{1}{2} \sum a_{\nu} x_{\nu}^2 + w_{\nu} / x_{\nu}^2$.
For $n=2$ and the linear potential we obtain the classical spherical
pendulum on $S^2$, see e.g.~\cite{cushman97}.
For $n=2$ with the quadratic potential we obtain the classical Neumann
system on $S^2$~\cite{neumann56}; the degenerate case with $a_1 = a_2$
(``quadratic spherical pendulum'') has been studied
in~\cite{bates93, efstathiou04}.

In this paper we are concerned with the cases in which ${\bf A}$ has
multiple eigenvalues, so that there is additional rotational symmetry
among the groups of co-ordinate axes $x_{\nu}$ with the same $a_{\nu}$.
Let $b_{\sigma}$, $\sigma=0, \dots, \ell$, denote the values of different
coefficients $a_{\nu}$ where $b_{\sigma}$ has multiplicity $m_{\sigma}$. 
We arrange groups of equal $a_{\nu}$ to be labeled consecutively and 
assume they are sorted by size $b_0 < \ldots < b_\ell$.
Generally we use greek indices $\nu, \kappa, \ldots$ ranging from $0$ to~$n$
and $\sigma, \tau, \ldots$ ranging from $0$ to~$\ell$, while latin indices
$i, k, \ldots$ run through the index sets $I_{\sigma}$ that contain the
$m_{\sigma}$ indices which have the same coefficient $b_{\sigma} = a_i$
for $i \in I_{\sigma}$.
Then the potential can be written as
\begin{displaymath}
   V =  b_0 V_0 + b_1 V_1 + \ldots + b_{\ell} V_{\ell}, \qquad
   V_{\sigma} = \frac{1}{2} \sum_{i \in I_{\sigma}}^{m_{\sigma}} x_i^2  \,.
\end{displaymath}
The potential is invariant under the symmetry group
\begin{displaymath}
   G   =   O(m_0) \times O(m_1) \times \ldots \times O(m_{\ell})
\end{displaymath}
where $O(m_{\sigma})$ acts on the co-ordinates summed over in $V_{\sigma}$: 
the sum of their squares remains constant.
The group $G$ describes the ambiguity for the choice of an orthonormal basis
of eigenvectors of~${\bf A}$.
Indeed, the co-ordinate system chosen above is obviously not unique, but
any $G$--action on it gives another admissible co-ordinate system (possibly with opposite orientation) 
in which the Hamiltonian assumes the same form.
Under the rotations from $G$ the velocities transform (by tangent lift) in
the same way as the co-ordinates, so the kinetic energy remains constant
as well. 

The integrals \eqref{eqn:genIntegrals} become singular in the degenerate case.
The integrals in the degenerate case were found in~\cite{Zhangju92}.
For each set of $m_{\sigma}$ equal coefficients $a_i = b_{\sigma}$,
$i \in I_{\sigma}$ the integrals are the angular momenta $L_{ik}$ where
$i, k \in I_{\sigma}$ and the function 
\begin{equation} \label{eqn:integralsdegenerate}
    F_{\sigma} = \sum_{i \in I_{\sigma}}^{m_{\sigma}} x_i^2
    + \sum_{\tau \neq \sigma}^{\ell}
    \sum_{k \in I_{\sigma}}^{m_{\sigma}}
    \sum_{l \in I_{\tau}}^{m_{\tau}}
    \frac{L_{kl}^2}{b_{\sigma}-b_{\tau}} \,.
\end{equation}
Assembling the angular momenta into
\begin{displaymath}
   W_{\sigma} =
   \sum_{i<k \in I_{\sigma}}^{\frac{1}{2} m_{\sigma}(m_{\sigma} - 1)}
   L_{ik}^2
\end{displaymath}
we can express the Hamiltonian on $T^*S^n$ as
\begin{displaymath}
   H = \frac{1}{2} \sum_{\sigma = 0}^{\ell}
   b_{\sigma} F_{\sigma} + W_{\sigma} \,.
\end{displaymath}
Considering the case with equal coefficients $a_i = b_\sigma$, $i\in I_\sigma$
as the limit (denoted by $a \to b$ below)
of the generic case with all $a_i$ distinct the integrals $\tilde F_\nu$ of the
generic case diverge. However, the following combinations remain finite,
\begin{displaymath}
   \lim_{a \to b} (a_\nu - a_\mu) \tilde F_\nu = L^2_{\mu\nu}, \qquad
   \lim_{a \to b}  \sum_{k \in I_\sigma} \tilde F_k = F_\sigma \,,
\end{displaymath}
and the relation $\sum_{\sigma = 0}^\ell F_\sigma = C_1 = 1$ on~$S^n$
is inherited.

\subsection{Basic Dynamics}
\label{basicdynamics}

Some simple observations can be made directly from the 
differential equations of the Neumann system. 
Understanding this local information helps to put together the
global picture later on.

Observe that the differential equations \eqref{eqn:Newton} have
the simple form  $\ddot \xx =  {\bf C} \xx$ with the diagonal
matrix ${\bf C} = \mathrm{diag}( (\lambda - a_{\nu} )) $.
This may appear to be linear, but in fact $\lambda = 2 V - 2 T$
depends on all the variables.
Still, for any choice of subset of $d$ indices
$I = \{ \nu_1, \dots, \nu_d \}$ the $2d$ dimensional invariant
subsystem $x_{\nu} = y_{\nu} = 0$, $\nu \not \in I$,  is again a
Neumann system of lower dimension. 

In the simplest case $d=1$ the invariant subsystem consists of
two equilibrium points at $(\xx,\yy)  = (\pm \e_{\nu},  0)$.
The Hessian of the potential $V$ in local co-ordinates
(which are $\xx$ with $x_{\nu}$ omitted) near the equilibrium has
eigenvalues $2(a_{\kappa} - a_{\nu})$, $\kappa \neq \nu$. 
Denote the number of $a_{\kappa} > a_{\nu}$ by $m_{\nu}^u$ and 
the number of $a_{\kappa} < a_{\nu}$ by $m_{\nu}^s$.
Then the spectrum of the equilibrium at $\pm \e_{\nu}$ has $m_{\nu}^u$
nonzero real eigenvalues and $m_{\nu}^s$ nonzero imaginary eigenvalues.
For the coefficients $b_{\sigma}$ with multiplicity $m_{\sigma}$ there
is a sphere $S^{m_{\sigma}-1}$ of equilibrium points of dimension
$m_{\sigma}-1$ given by $V_{\sigma} = \frac{1}{2}$.
All of them are equilibrium points of (\ref{eqn:Newton}) because
$\lambda = b_{\sigma}$ exactly balances the force for all nonzero
co-ordinates.
These non-isolated equilibrium points are reflected by the number
$m_{\sigma} - 1$ of zero eigenvalues.

For $d=2$ there is a Neumann subsystem with 1 degree of freedom in the
plane with the chosen indices.
The potential on the circle obtained by intersecting the plane with
the sphere $C_1 = 1$ is periodic, and either has two minima and two
maxima (like ``twice'' a pendulum), or is constant when the coefficients
$a_{\nu}$ in this plane are equal. 
The types of motion inside this plane are easy to obtain, but the normal 
stability of these orbits is already non-trivial.
These $n(n+1)/2$ families of periodic orbits  in particular contain the
non-linear normal modes of the $n+1$ equilibria.

For $d=3$ we find a Neumann subsystem with 2 degrees of freedom. 
The generic motion takes place on $2$--tori $\T^2$ when the coefficients
of the three planes are not all equal.
In this case there exist isolated periodic orbits on the energy shell
which are not inside an invariant plane with $d=2$.
When all three coefficients are equal the potential is constant on the
sphere, and the system defines the geodesic flow on~$S^2$.
In that case all the orbits are closed, and in fact great circles on $S^2$.

The hierarchy of invariant subsystems continues up to $d=n$.
Each invariant set has a normal stability, and invariant manifolds
associated to that normal stability.
Together with the invariant tori they form an intricate system of invariant
submanifolds that together foliate the phase space.
The foliation for the non-degenerate Neumann system has been studied
in~\cite{dullin01}.
Interestingly, the foliation becomes more complicated when additional 
symmetry is introduced. 
This additional structure is described in this paper.

Now we focus on those invariant subspaces for which the above index set~$I$
is equal to an index set~$I_{\sigma}$ of equal coefficients, given by
\begin{equation} \label{eqn:subspace}
   \left\{ (\xx, \yy) \in T^*S^n :
   x_{\nu} = y_{\nu} = 0 \, \forall_{\nu \notin I_{\sigma}}
   \right\} \,.
\end{equation}
Consider an initial condition $(\xx_0, \yy_0)$ with $\yy_0 \neq 0$
inside this invariant subspace. 
Choose a new orthogonal co-ordinate system which has among its unit vectors  
the directions of $\xx_0$ and $\yy_0$.
This is always possible by an orthogonal transformation because
$\xx_0$ and~$\yy_0$ are orthogonal.
Moreover, the 2nd order differential equation \eqref{eqn:Newton} is
invariant under such a rotation, because we only need to rotate within
the subset of co-ordinates $I_{\sigma}$ for which there are equal
coefficients.
We already noted that the symmetry group $G$ originates from this freedom
of choice of co-ordinates.
After this co-ordinate transformation, only two of the $m_{\sigma}$
co-ordinates are involved in the dynamics, while the remaining
$m_{\sigma} - 2$ co-ordinates remain constant.
In this sense locally nothing new happens when $m_{\sigma} \geq 3$.
However, globally, the geometry does become more interesting for
$m_{\sigma} \geq 3$: a different initial condition gives a solution that
differs by a rotation.
The set of all such solutions is given by the set of all oriented
$2$--planes in $\R^{m_{\sigma}}$.
This is a Gra{\ss}mannian manifold and it appears as a coadjoint
orbit below.

Interestingly, most of the above analysis can still be applied even when
the initial condition $(\xx_0, \yy_0)$ is {\em not} inside an invariant
subspace~\eqref{eqn:subspace}.
We then have co-ordinates in different $I_{\sigma}$ and all these can
again be simplified by a rotation, so that only two co-ordinates in
each subsystem are non-constant. 
The difference is that the Lagrange multiplier in \eqref{eqn:Newton}
depends on all the co-ordinates and momenta, and therefore is not a
constant anymore.
In this way, after the co-ordinate transformation all the remaining
equations are coupled in a non-trivial way.

When at least one $m_{\sigma} \geq 3$ the system becomes superintegrable,
cf.~\cite{nekhoroshev72,fomenko78,evans90,kibler90,Fasso04}, having more
integrals than degrees of freedom.
This makes the phase space a ramified torus bundle, with regular
fibres~$\T^{\ell + \tilde{\ell} + 1}$.
Here $\ell + 1$ is the number of distinct eigenvalues of~${\bf A}$ and
$\tilde{\ell} + 1$ is the number of groups of multiple eigenvalues.
If $m_{\sigma} \leq 2$ for all $\sigma = 0, \ldots, \ell$ then
$\ell + \tilde{\ell} + 1 = n$ and the regular fibres of the ramified
torus bundle are Lagrangian tori.
Under regular reduction the fibres $\T^{\ell + \tilde{\ell} + 1}$
are mapped to $\T^{\ell}$ --- the generic reduced motion takes place
on $\ell$--tori.
Hence $\ell$ is the number of frequencies needed to describe a generic
reduced motion.
For each of the $\tilde{\ell} + 1$ groups~$I_{\sigma}$ with
$m_{\sigma} \geq 2$ the full system turns out to have one additional
frequency.

In the following we are going to make the above more precise and describe
the associated global geometry.
Our treatment of the degenerate Neumann system heavily relies on the
reduction with respect to the symmetry group
\begin{displaymath}
   G = O(m_0) \times \ldots \times O(m_{\ell})
\end{displaymath}
with subgroup
\begin{displaymath}
   \widetilde{G} = O(m_0) \times \ldots \times O(m_{\tilde{\ell}}) \, , \quad
   G / \widetilde{G} \cong \Z_2^{\ell - \tilde{\ell}}
\end{displaymath}
whence we now recall some facts about the cotangent lifted group action
of the orthogonal group $O(m)$.
We shall later see that understanding this group action is the key to 
an efficient description of the symmetry reduction in the
degenerate Neumann system.
The following section also serves as a quick introduction to (singular)
reduction and to the geometry of superintegrable systems.

\section{Orthogonal Group Action}
\label{orthogonalgroupactions}

Consider the orthogonal group $O(m)$ of invertible $m\times m$ matrices
$g$ which preserve the Euclidean scalar product in $\R^m$, so that for
$\xx \in \R^m$ we have $\scap{g \xx}{ g \xx} = \scap{\xx}{\xx}$.
Let the group $G = O(m)$ act on $T^*\R^m = \R^{2m}$ with 
co-ordinates $(\xx,\yy)$ and symplectic structure
$\Omega = \dee \xx \wedge \dee \yy$.
The rotation $g \in O(m)$ gives a point transformation by left
multiplication, the induced action on the cotangent bundel $T^*\R^m$
is obtained by cotangent lift.
Since $(g^t)^{-1} = g$ we find the symplectic action
\begin{equation} \label{eqn:Oaction}
\begin{array}{cccc}
   \Phi : & O(m) \times T^*\R^m & \longrightarrow & T^*\R^m  \\
   & (g,  \xx,  \yy) & \mapsto & \Phi_g(\xx, \yy) = (g \xx, g \yy) \,.
\end{array}
\end{equation}
Hence, the orthogonal group acts by rotating co-ordinates and conjugate
momenta in the same way.
The general $\Phi$--invariant Hamiltonian is a function of the invariants
of this group action, whence it can be expressed as a function
\begin{equation}  \label{eqn:HamCF}
    H(\xx, \yy)  =
    \hat{H}( \scap{\xx}{\xx}, \scap{\yy}{\yy}, \scap{\xx}{\yy} )
\end{equation}
of the three basic invariants of the action~$\Phi$.
In other words, these three basic invariants form a Hilbert basis, 
and since there are no constraining syzygies between these
invariants this is a free Hilbert basis.
A standard example is the motion in a central force field, for
which the Hamiltonian reads
\begin{displaymath}
   H (\xx, \yy) = \frac{1}{2} \scap{\yy}{\yy} + U( \scap{\xx}{\xx} ) 
\end{displaymath}
with radial potential~$U$.

\subsection{Momentum mapping}
\label{sub:momentummapping}

By Noether's theorem each one-parameter symmetry leads to a conserved quantity.
For multi-parameter groups like~$O(m)$ the integrals of motion can be assembled
into the momentum mapping of the group action.

\begin{lemma}  \label{lem:momOm}
   For $m \geq 2$ the action of $G = O(m)$ on $T^*\R^m$ given by
   $\Phi_g(\xx, \yy) = (g\xx, g\yy)$ 
   has an $\Ad^*$--equivariant Momentum mapping given by
   \begin{equation} \label{eqn:MomMapOm}
   \begin{array}{cccc}
      \J : & T^*\R^{m} & \longrightarrow & \oo(m)^*  \\
      & (\xx,\yy) & \mapsto &  \xx \otimes \yy - \yy \otimes \xx \,.
   \end{array}
   \end{equation}
   Any $0 \neq \mu \in \oo(m)^*$ is a weakly regular value, which is
   regular for $m = 2, 3$.
\end{lemma}

\noindent
We recall that Marsden--Weinstein reduction~\cite{marsden74} works for
weakly regular momentum values.
A value of a differentiable mapping is called weakly regular if the tangent
space at every point of the pre-image is given by the kernel of the Jacobian
at that point.
Lemma~\ref{lem:momOm} is true in general for faithful cotangent lifted
matrix Lie group actions, see, e.g., 12.1.1 in~\cite{MR}.
For the convenience of the reader and in order to introduce 
our notation we give a direct proof.

\begin{proof}
Let $\xi \in \oo(m)$ be the infinitesimal generator of $g = \exp \xi$. 
The symmetry $\Phi_g$ is the time one mapping of the flow 
of the vector field $X_\xi = ( \xi \xx, \xi \yy)$.
This vector field has the Hamiltonian $J_\xi =  \scap{\xi \xx}{\yy}$.
The Momentum mapping $\J : T^*\R^m \longrightarrow \oo(m)^*$ is defined by
\begin{displaymath}
  \langle \J(\xx,\yy), \xi\rangle_{\oo} = J_\xi(\xx,\yy) \,,
\end{displaymath}
where the pairing
$ \langle A, B \rangle_{\oo} = \frac{1}{2} \trace A B^t$ 
is used in $\oo(m)$.
Using $\trace (\xx \yy^t) = \scap{\xx}{\yy}$ equation~\eqref{eqn:MomMapOm}
is easily verified. 
The components of $\J$ are of course the ordinary 
angular momenta $L_{ik} = x_i y_k - x_k y_i$, and $L_{ik}$ generates
a rotation in the $ik$--plane.

When the symmetry $\Phi_g$ is applied to $(\xx, \yy)$ 
the corresponding image of the Momentum mapping is transformed into 
\begin{equation} \label{eqn:preAd}
    (\J \circ \Phi_g)(\xx,\yy)  =
    (g\xx) (g\yy)^t - (g\yy)(g\xx)^t  =
    g( \xx \yy^t - \yy\xx^t)g^t \,.
\end{equation}
This in turn is just the (co)adjoint action of the group on the
(dual of the) algebra and reads $\Ad^*_{g^{-1}} \mu = g \mu g^{-1}$
for the orthogonal group.
Hence $\J$ is $\Ad^*$--equivariant,
\begin{displaymath}
     \Ad^*_{g^{-1}} \circ \J =   \J \circ \Phi_g  \,.
\end{displaymath}
The value $\mu = 0$ is not weakly regular because $\J^{-1}(0)$ is not a
smooth manifold (instead, it is a cone).
For $\mu \neq 0$ the set $\J^{-1}(\mu)$ is a submanifold of
$T^*\R^m$, and the tanget space
\begin{displaymath}
   T_{(\xx_0, \yy_0)} \J^{-1}(\mu) = \left\{
   (\alpha \xx_0 + \beta \yy_0, \gamma \xx_0 - \alpha \yy_0) \in
   T_{(\xx_0, \yy_0)} (T^*\R^m) : \alpha, \beta, \gamma \in \R
   \right\}
\end{displaymath}
is identical to the kernel of $D \J (\xx_0, \yy_0)$ at any point
$(\xx_0, \yy_0) \in \J^{-1}(\mu)$.
When $m=2,3$ the rank of $D\J$ is full when $\mu \not = 0$, 
and so the value is regular.
\end{proof}

\subsection{Group orbits}
\label{sub:grouporbits}

All symmetry transformations applied to a point $p \in T^*\R^m$ produce
the group orbit $\, G \cdot p = \{ \Phi_g(p) : g \in G \} \,$ of $p$.
It is isomorphic to the quotient of the full group $G$ by the isotropy
subgroup $G_p = \{ g \in G : \Phi_g p = p \}$; the `motions' that fix $p$
have to be divided out.
For convenience we define $O(0)$ to be one point.

\begin{lemma}  \label{lem:Stiefel}
   Let $p=(\xx,\yy) \in \J^{-1}(\mu)$, then the $G$--orbit of $p$ under
   the action $\Phi$ is isomorphic to
   \begin{displaymath}
       G/G_p = O(m)/O(m-d) =  V_{m,d}\qquad \text{(Stiefel manifold)}
   \end{displaymath}
   where $d = \dim \mathrm{span} \{\xx, \yy\}$. 
\end{lemma}

\noindent
The Stiefel manifold $V_{m,k} = O(m)/O(m-k)$ 
is the set of $k$--dimensional orthonormal frames in $\R^m$. 

\begin{proof}
The isotropy groups $G_p$ are easily determined. 
There are three cases,
\begin{itemize}
   \item $d = 2: \xx \not\parallel \yy \Longrightarrow G_p = O(m-2)$:
      the generic case,
   \item $d = 1: \xx \parallel \yy \Longrightarrow G_p = O(m-1)$:
      the singular case,
   \item $d = 0: \xx = \yy = 0 \Longrightarrow G_p = O(m)$:
      the trivial case.
\end{itemize}
In the generic case the `motion' takes place in the $2$--dimensional
plane spanned by $\xx$ and~$\yy$.
The isotropy group is the subgroup of rotations that fixes this plane.
In the singular case the `motion' takes place on a line, and $G_p$ is the
subgroup of rotations that fix this line. 
The group orbits of points $p$ in these three cases are diffeomorphic to
\begin{itemize}
   \item $d = 2 : G/G_p = O(m)/O(m-2) = V_{m,2}$:
      Stiefel manifold, $\dim V_{m,2} = 2m-3$,
   \item $d = 1 : G/G_p = O(m)/O(m-1) = V_{m,1} = S^{m-1}$:
      Sphere, $\dim S^{m-1} = m-1$,
   \item $d = 0 : G/G_p = \{ \id \}$:
      Point $(\xx, \yy) = (0,0)$.
\end{itemize}
\end{proof}

\noindent
Now fix the momentum $\mu = \J(p)$. When $p' = \Phi_g (p)$ runs through
the orbit $G \cdot p$ of~$p$ then $\mu' = \J(p') = \Ad^*_{g^{-1}} (\mu)$
by equivariance.
The set of all such momenta is the group orbit $G \cdot \mu$ of~$\mu$ 
under the coadjoint action $\Ad^*$.
It is isomorphic to $G/G_{\mu}$ where the isotropy subgroup is given by
$G_{\mu} = \{ g \in G: \Ad^*_{g^{-1}} \mu = \mu \}$.

\begin{lemma}  \label{lem:Grassmann}
   The $G$--orbit of $\mu \neq 0$ under the coadjoint action
   $\Ad^*_g \mu = g\mu g^{-1}$ is isomorphic to
   \begin{displaymath}
       G/G_{\mu} =  O(m)/ ( SO(2) \times O(m-2) ) = G_{m,2}
       \qquad \text{(Gra{\ss}mann manifold).}
   \end{displaymath}
\end{lemma}

\noindent
The Gra{\ss}mann manifold $G_{m,k} = O(m)/( SO(2) \times O(m-k))$
is the set of $k$--dimensional oriented subspaces of~$\R^m$ and
can alternatively be defined as $SO(m)/( SO(2) \times SO(m-k))$.

\begin{proof}
First consider the case $m=2$.
Let $e = \begin{pmatrix} 0 & 1 \\ -1 & 0 \end{pmatrix} $,
and for $a \in O(2)$ require $a e = e a$.
This implies that actually $a \in SO(2)$.
For $m \geq 3$ note that any nonzero rank~$2$ antisymmetric matrix can be
brought into the form $\mathrm{diag}( e, 0)$.
Hence it is enough to consider this special
$\mu = \mathrm{diag}( e, 0) \in \oo(m)^*$.
The invariance $g \mu g^t = \mu$ then requires, writing $g$ in block form,
\begin{displaymath}
   \begin{pmatrix} a & b\\ c & d    \end{pmatrix}
   \begin{pmatrix}  e & 0 \\ 0 & 0    \end{pmatrix}
= 
   \begin{pmatrix}  e & 0 \\ 0 & 0    \end{pmatrix}
   \begin{pmatrix} a & b\\ c & d    \end{pmatrix} \,.
\end{displaymath}
In addition to $a\in SO(2) $ this gives $b = c = 0$.
Since $g \in O(m)$ the other block $d \in O(m-2)$.
Conversely, any $g$ of the form $\mathrm{diag}( SO(2), O(m-2) )$ leaves
$\mu = \mathrm{diag}( e, 0)$ invariant, hence
$G_{\mu} =  SO(2) \times O(m-2)$. 
The coadjoint orbit  $G \cdot \mu \cong G / G_{\mu}$ is the
Gra{\ss}mannian $G_{m,2} = O(m)/( SO(2) \times O(m-2) )$,
the set of oriented $2$--planes of $\R^m$.
\end{proof}

\noindent
Let $H$ be a $\Phi$--invariant Hamiltonian,
$H \circ \Phi_g = H$ for any $g \in O(m)$.
Then by Noether's theorem the flow~$\varphi_H^t$ of the Hamiltonian
leaves $\J$ invariant, $\J \circ \varphi_H^t = \J$.
Hence, each component of $\J$ is a constant of motion.
All the components $L_{ik} = x_i y_k - x_k y_i$ are integrals of motion,
but only $2m-3$ of them are independent, $\rank \J = 2m-3$. 
Even when all components are independent (as for $m=3$) 
they do not all commute. 

As an example consider the general Hamiltonian~\eqref{eqn:HamCF} with
$m=3$, or the particular case of the motion in a central force field.
As usual for $m=3$ the dual of the Lie algebra $\oo(3)^*$ is identified
with $\R^3$.
The generic initial condition $(\xx, \yy)$ satisfies
$\xx \times \yy \neq 0$ so that the motion takes place in the plane
through the origin perpendicular to the angular momentum $\xx \times \yy$.
This plane is fixed by the reflection $g \in G_p = O(1)$.
The group orbit of a point $p = (\xx, \yy)$ is $O(3)/G_p = SO(3)$.
The angular momentum $\mu$ is not constant along this group orbit.
It is only constant when $g \in G_{\mu} = SO(2) \times O(1)$, i.e.\ when
we fix the axis given by the angular momentum.
When other rotations act on $\mu$ the angular momentum takes on all
possible values in the coadjoint orbit $O(3)/G_{\mu} = S^2$. 
This is the space of oriented planes in $\R^3$. 
In the singular case the angular momentum vanishes and the line of motion
is fixed by rotations about this line, while in the trivial case the origin
is fixed by every rotation.
In both cases the coadjoint orbit is not the correct notion for the
analysis of the geometry since $\mu = 0$, see the singular reduction below.

\subsection{Regular reduction}
\label{sub:regularreduction}

The regular reduction procedure is to fix a weakly regular
momentum value $\mu \in \gg^*$ and factor out the $G_{\mu}$--action
from the set $\J^{-1}(\mu)$ of points that have fixed
momentum~$\mu$ and form the quotient 
\begin{displaymath}
   P_{\mu} = \J^{-1}(\mu) / G_{\mu} \,.
\end{displaymath}
For an abelian group we have $G_{\mu} = G$, while in the non-commutative
case only dynamics associated to group elements $g$ whose coadjoint
action preserves $\mu$ is factored out, hence $G_{\mu}$.
The reduced symplectic manifold $P_{\mu}$ then is the
space of $G_{\mu}$--orbits in $\J^{-1}(\mu)$.
The dimension of the reduced space is given by
\begin{displaymath}
   \dim P_{\mu} = \dim P - 2 \dim G/G_p + \dim G/G_{\mu}
\end{displaymath}
where $P$ is the original phase space.
Thus, for $G = O(m)$ and $P = T^*\R^m$ we find
\begin{displaymath}
    \dim P_{\mu} = 2m - 2 (2m-3) + (2m-4) = 2 \,.
\end{displaymath}
The regular reduction for the $O(m)$--action $\Phi$ is made effective by
giving concrete models for the abstract reduced phase spaces.

\begin{lemma}  \label{lem:RegRedOm}
   The reduced space $P_{\mu}$ for $\mu \neq 0$ is diffeomorphic to
   the open half plane in $\R^2$ with global symplectic co-ordinates
   \begin{displaymath}
       \xi = || \xx || > 0 \,, \quad
       \eta = \scap{\xx}{\yy}/||\xx|| \in \R
   \end{displaymath}
   so that $[\xi, \eta] = 1$.
   The $\Phi$--invariant Hamiltonian  \eqref{eqn:HamCF} becomes
   \begin{equation} \label{eqn:hatH}
       \tilde{H}(\xi, \eta, J) =
       \hat{H}(\xi^2, \eta^2 + \frac{J^2}{\xi^2}, \xi \eta)
       \,,
   \end{equation}
   where the total angular momentum $J = ||\J||_{\gg^*} = \sqrt{W}$
   is given by
   \begin{displaymath}
       W  = \frac{1}{2} \trace \J \J^t = \sum_{1 \leq i < k \leq m} L_{ik}^2
       \,.
   \end{displaymath}
\end{lemma}

\begin{proof}
Since $\mu \neq 0$ we have $\xi > 0$.
The half plane $\xi > 0$ is of course diffeomorphic
(even symplectomorphic) to $\R^2$.
The co-ordinates $(\xi, \eta)$ are symplectic since
\begin{displaymath}
    [ \xi, \eta ] = \sum_{i=1}^m
    \frac{\partial \xi}{\partial x_i} \frac{\partial \eta}{\partial y_i}
    = \frac{1}{||\xx||^2} \sum_{i=1}^m x_i^2 = 1 \,.
\end{displaymath}
The identity
\begin{equation} \label{eqn:Jxy}
   W = ||\xx||^2  ||\yy||^2 - \scap{\xx}{\yy}^2 = (||\yy||^2 - \eta^2) \xi^2
\end{equation}
allows to eliminate $||\yy||^2$ from $\hat{H}$ and drop it to $\tilde{H}$
on~$P_{\mu}$.
The fibre of the reduction mapping from the submanifold of constant total
angular momentum~$j$,
\begin{displaymath}
   J^{-1}(j) = \left\{ (\xx, \yy) \in T^*\R^m : J(\xx, \yy) = j \right\} 
\end{displaymath}
to $(\xi, \eta)$ is the Stiefel manifold $V_{m,2}$ found before. 
\end{proof}

\noindent
For $m=2$ the above reduction is equivalent to introducing the radius of
polar co-ordinates in the plane. For fixed nonzero angular momentum 
radius zero is impossible, the line $\xi = 0$ is not part of the 
reduced space, and the reduction is regular.

\begin{remark}  \label{rem:polynomial}
The function $W$ is a polynomial in the basic invariants, as opposed to
$J = \sqrt{W}$.
The square root is not differentiable at $0$ and this difference between 
$J$ and $W$ becomes important in the singular reduction later on.
From an algebraic point of view $W$ is more fundamental, but
from the symplectic point of view $J$ is more fundamental, since $J$ defines
an action outside $\mu = 0$, i.e.\ $J$ is a Hamiltonian with
$2 \pi$--periodic flow when $\mu \neq 0$.
In case $m=2$ one may choose to reduce only the $SO(2)$--action instead
of the full $O(2)$--action.
This yields a fourth basic invariant~$L_{12}$ and the Hilbert basis
ceases to be free as the basic invariants become related by the
syzygy $L_{12}^2 + \scap{\xx}{\yy}^2 = ||\xx||^2  ||\yy||^2$, in
particular $W = L_{12}^2$.
The reason $m=2$ is special  is that $J$ can be redefined as $J = \J = L_{12}$,
which {\em is} polynomial and provides a smooth passage through the
critical value $\mu = 0$.
\end{remark}

\noindent
Since we have many more integrals (namely $m(m-1)/2$, of which $2m-3$
are independent) the Hamiltonian $H$ is superintegrable for $m \geq 3$.
For a weakly regular value of the Energy--Momentum mapping 
\begin{equation} \label{eqn:EM}
   (H, \J) : T^*\R^m \longrightarrow \R \times \oo(m)^*
\end{equation}
the pre-image is a two-dimensional set.
It is a two-torus when compact, and there are two actions associated to it.
One action is the total angular momentum~$J$, the other comes from the
reduced Hamiltonian.
All the other $2m-4$ dimensions do not take part in the dynamics.
The semi-local situation is described by the following result.

\begin{lemma} {\em (Generalized action-angle co-ordinates)} \label{lem:nekho}
   The symplectic form on $T^*\R^m$ locally near a weakly regular
   value of the Energy--Momentum mapping \eqref{eqn:EM} decomposes into 
   \begin{displaymath}
      \Omega = \Omega_{\mu} + \dee \psi \wedge\dee J + \Omega_{G \cdot \mu}
   \end{displaymath}
   where $\Omega_{\mu} = \dee \xi \wedge \dee \eta$ is the symplectic
   structure on the reduced space $P_{\mu}$, $\psi$ is conjugate to the
   action $J$, and $ \Omega_{G \cdot \mu}$ is the symplectic
   structure on the coadjoint orbit space~$G_{m,2}$.
   The reduced Hamiltonian \eqref{eqn:hatH} only depends on the
   co-ordinates $(\xi, \eta, J)$.
   For a generic invariant Hamiltonian a system of generalized action-angle
   co-ordinates in the sense of~\cite{nekhoroshev72} is obtained by
   the action-angle co-ordinates of the reduced one-degree-of-freedom
   system in $(\xi,\eta)$ and by $(\psi, J)$ together with any set of
   symplectic co-ordinates on~$G_{m,2}$.
\end{lemma}

\begin{proof}
Let us first of all check that $J$ is an action, i.e.\ that its flow 
is $2\pi$--periodic. 
Note that using \eqref{eqn:Jxy}
\begin{displaymath}
   \frac{1}{2} \nabla_{\yy} J^2 =  
   {||\xx||^2} \yy -  \scap{\xx}{\yy} \xx
   \quad \mbox{and} \quad
   \frac{1}{2} \nabla_{\xx} J^2 =  
   {||\yy||^2} \xx  - \scap{\xx}{\yy} \yy \,.
\end{displaymath}
For each component Hamilton's equations for $J$ therefore are
\begin{displaymath}
   \begin{pmatrix} \dot x_i \\ \dot y_i \end{pmatrix}
   = S M 
   \begin{pmatrix}  x_i \\  y_i \end{pmatrix}, \qquad
   S = \begin{pmatrix} 0 & 1 \\ -1 & 0 \end{pmatrix}, \quad
   M = \frac{1}{J}
   \begin{pmatrix}
      ||\yy||^2 & -\scap{\xx}{\yy} \\
      - \scap{\xx}{\yy} & ||\xx||^2
   \end{pmatrix} \,.
\end{displaymath}
Note that $SM$ is an involution, $(SM)^2 = -\id$, which follows from 
$\det M = 1$, see \eqref{eqn:Jxy}.
Hence the eigenvalues of $SM$ are $\pm {\rm i}$, and the
$2\pi$--periodic flow is 
\begin{displaymath}
   \Phi_J^t (x_i, y_i) =
   \left( \id \cos t + SM  \sin t \right ) 
   \begin{pmatrix} x_i \\ y_i \end{pmatrix} \,,
\end{displaymath}
for every $i = 1, \ldots, m$.
The projection of the orbit to configuration space (or momentum space)
is a circle in the plane spanned by the initial $\xx$ and $\yy$, with
radius $||\xx||$ (or $||\yy||$). 
This follows because $(-\scap{\xx}{\yy} \xx + ||\xx||^2 \yy)/J $ is
orthogonal to $\xx$ and with length $||\xx||$, similarly for the second
row of $SM$ and $\yy$.
The flow of $J$ yields the angle $\psi$ conjugate to~$J$.
The computation of the 2nd action depends on the Hamiltonian.

Since the reduced Hamiltonian has one degree of freedom and depends on
the action $J$ all the other co-ordinates belong to the symplectic
manifold of co-ordinates that do not appear at all in the Hamiltonian.
This symplectic manifold of ignorable symplectic co-ordinates is given by
removing the periodic flow of $J$ from the fibre of the reduction mapping,
$V_{m,2}/S^1  = G_{m,2}$.
Hence, the coadjoint orbit $G \cdot \mu \cong G_{m,2}$ with its symplectic
structure (see, e.g., \cite{MR}) is this symplectic manifold.
\end{proof}

\noindent
Generalized action-angle co-ordinates in the sense of~\cite{nekhoroshev72}
are obtained from this construction when $H$ is a generic $O(m)$--invariant
function, so that the generic motion has two frequencies.
For special $H$ almost every orbit is periodic, and in this case the
construction needs to be modified. 

In the picure developed in~\cite{Fasso96, Fasso04} there is a two-dimensional 
meadow of actions (which globally can be replaced by the image of the
Energy--Casimir mapping, see below), on which there are flowers whose petals
are two-tori parametrised by the angles conjugate to the actions, and the
centre of the flower carrying the petals is the coadjoint orbit $G \cdot \mu$.
More precisely, the pre-image
\begin{displaymath}
   {\cal F}_j  = J^{-1}(j) = 
   \bigcup_{\tilde \mu \in G \cdot \mu} \J^{-1} (\tilde \mu)
   \,, \quad j = || \mu ||
\end{displaymath}
is mapped onto the centre of the flower, $\J({\cal F}_j) = G_{m,2}$.
Therefore the flower is a bundle over~$G_{m,2}$, and following
Nekhoroshev~\cite{nekhoroshev72} the fibre of this mapping is~$\T^2$.
Since the Nekhoroshev actions are only locally defined, some
flowers are still missing in the global picture.

\subsection{Singular reduction}
\label{sub:singularreduction}

All we have said so far is only valid for weakly regular values.
To understand what happens near $\mu = 0$ we need to use singular reduction. 
Note that setting $J = 0$ in the reduced Hamiltonian $\hat{H}$~\eqref{eqn:hatH}
does correctly describe reduced motion on a line in configuration space. 
However, if fails to describe how the neighbouring reduced systems with two
degrees of freedom and $\mu \neq 0$ limit to this special case.

\begin{lemma}  \label{lem:SingRedOm}
   The reduced phase space $P_{\mu}$ is the image of the momentum level
   set $\J^{-1}(\mu)$ under the Hilbert mapping
   \begin{displaymath}
   \begin{array}{cccc}
      {\chi} : & T^*\R^m & \longrightarrow & \R^3  \\
      & (\xx, \yy) & \mapsto & (V, T, S)
   \end{array}
   \end{displaymath}
   where we put $V = \frac{1}{2} ||\xx||^2$, $T = \frac{1}{2}||\yy||^2$
   and $S = \scap{\xx}{\yy}$.
   When $\mu \neq 0$ then $P_{\mu} \simeq \R^2$, given by the sheet
   $V \geq 0$, $T \geq 0$ of the two-sheeted hyperboloid
   \begin{equation} \label{eqn:JVTS}
      2 VT - \frac{1}{2} S^2 = \frac{1}{2} j^2,
   \end{equation}
   where $j = ||\mu||_{\gg^*}$
   For $\mu = 0$ it is the half cone $2VT = \frac{1}{2}S^2$,
   with $V \geq 0, T \geq 0$.
   The fibres of $\chi$ are the Stiefel manifolds $V_{m,d}$ with $d=2$
   in the regular case and $d=1$ or $d=0$ for $\mu = 0$.
   The reduced phase space $P_{\mu}$, $\mu \neq 0$ fixed, is the
   symplectic leaf of the Poisson structure on $\R^3$ with
   co-ordinates $(V,T,S)$ satisfying the bracket relations
   \begin{displaymath}
      [ V, T ] = S , \quad 
      [ V, S ] = 2 V , \quad
      [ T, S ] = -2 T
   \end{displaymath}
   of $\mathfrak{sl}(2, \R)$.
   The function $W=J^2$ on~$T^*\R^m$ is a Casimir of this bracket.
   The reduced Hamiltonian $\hat{H}$ is a function of $V$, $T$ and $S$ only,
   reading $\hat{H} = T + U(V)$.
\end{lemma}

\begin{proof}
For the computation of the brackets it is enough to compute e.g.,
$[ x_i^2, x_i y_i ] = 2 x_i^2$, etc.
The induced bracket on $\R^3$ with co-ordinates $(S,T,V)$ (for which
we use the same notation $[ .. \, , .. ]$) inherits the Jacobi identity.
While one may straightforwardly compute as well that $W$ is a Casimir
with respect to this bracket, this also follows from
$W : T^*\R^m \longrightarrow \R$ factoring through~$\oo(m)^*$ by means of
the $\Ad^*$--invariant expression $W = \sum L_{ik}^2$.
The symplectic leaves are given by the quadratic form \eqref{eqn:JVTS}, 
with the given relations.
For $W=0$ the reduced space is not a smooth manifold, but half of a cone
(in general such a reduced phase space is a semi-algebraic variety).
\end{proof}

\begin{remark}  \label{rem:SingRedByQuotient}
The singular reduction for $\mu = 0$ leads to a reduced system on a non-smooth
manifold, with a singular point at the tip of the cone $4 VT = S^2$.
An alternative description of this singular reduced phase space is
obtained by first restricting to any two-dimensional subspace that is
invariant under the dynamics of~$H$, e.g.\ the $(x_1, y_1)$--plane defined by
$x_i=y_i=0$ for $i \ge 2$. 
The $O(m)$--action $\Phi$ has a residual $\Z_2$--action on this plane given by
$\Phi(\pi) : (x_1,y_1) \mapsto (-x_1,-y_1)$.
Therefore the singular reduced phase space can also be viewed as $\R^2 / \Z_2$.
This is e.g.\ the half-plane $y_1 \ge 0$ with the boundary $y_1=0$ identified
with itself by means of $(x_1,0) \sim (-x_1,0)$, which again gives a half cone.
The fixed point of the residual $\Z_2$--action generated by $\Phi(\pi)$ is the
origin $x_1=y_1=0$ and projects to the singular point of the reduced phase
space.
\end{remark}

\noindent
The relation to the regular reduction in Lemma~\ref{lem:RegRedOm} is that 
$(\xi, \eta) = (\sqrt{2V}, S/\sqrt{2V})$ are symplectic co-ordinates on
any regular symplectic leaf.
It is, however, not a co-ordinate system for $W = 0$.
When $\mu = 0$ then $W=0$ and $\xx \parallel \yy$ for every
$p \in \J^{-1}(0)$.
Hence the fibre over the cone $P_0$ is $V_{m,1} = S^{m-1}$.
There is no dynamics on this sphere.
When $m = 2$ this $S^1$ is part of a regular invariant two-dimensional torus 
for almost all values of the energy.
For $m \geq 3$ this is not possible since in this case no $S^{m-1}$--bundle
over~$S^1$ is homeomorphic to an $m$--torus.
The upshot is that monodromy can be defined when $m=2$, but not when
$m \geq 3$, see below for more details.

Combining the previous results we can now describe the global meadow 
of actions, together with its maximal dynamical tori and symplectic manifolds
in which no dynamics takes place in the regular and singular cases.

\begin{lemma} {\em (Energy--Casimir Mapping)}
   The Energy--Casimir mapping
   \begin{displaymath}
      \EC = (H, W) : T^*\R^m \longrightarrow \R^2
   \end{displaymath}
   classifies the dynamics. 
   When it has full rank the fibre is a $\T^2$--bundle over $G_{m,2}$
   (or a $\T^1$--bundle over $V_{m,2}$).
   Where the rank is $1$ we have a relative equilibrium and the fibre
   over this point is $V_{m,2}$, which is a $\T^1$--bundle over $G_{m,2}$.
   In both cases the periodic flow of~$J$ gives an $S^1$--action such that
   the centre of the flower is $V_{m,2} / S^1 = G_{m,2}$.
   For $\mu=0$ the fibre of $\EC^{-1}(h, 0)$ is itself a ramified
   sphere bundle over the base space
   \begin{displaymath}
      \left\{ (V, T, S) \in \R^3 : \hat{H}(V, T, S) = h,
      2VT = {\textstyle \frac{1}{2}}S^2, V \geq 0, T \geq 0
      \right\}
   \end{displaymath}
   with regular fibres $V_{m,1} = S^{m-1}$ and singular fibre
   $V_{m,0}$ (a single point) over $(V, T, S) = 0$.
\end{lemma}

\begin{proof}
This follows from Lemmas~\ref{lem:Stiefel} and~\ref{lem:Grassmann}.
\end{proof}

\noindent
Note that we cannot say that the fibre over a regular point is
$\T^2 \times G_{m,2}$, since the $S^1$--bundle $V_{m,2}$ does not
posses a global section. 
In fact, for $m=3$ it is the Hopf fibration of $V_{3,2} = SO(3)$.

One way to view this classification result is to consider the two stage
process of first considering the Energy--Momentum mapping and then 
the mapping of sums of squares from momentum space to $\R$ that defines
the dynamically relevant Casimir. 

In $T^*\R^3$ all this comes down to a familiar picture.
The situation is simplified by identifying $\oo(3)^*$ with $\R^3$.
The orbit $G \cdot \mu \cong G/G_{\mu} = G_{3,2} = S^2$ is the
sphere with radius $j = || \mu ||$ in momentum space.
Already most of the group is involved in generating the coadjoint orbit, while
only the $S^1$ of rotations about the axis $\mu$ actually generates dynamics.
The motion in the configuration space~$\R^3$ takes place in the plane
orthogonal to $\mu$, which is the projection of the set $\J^{-1}(\mu)$.
The flow of $J$ generates an $S^1$ in this plane, and factoring this out
of $\J^{-1}(\mu)$ gives the reduced space $\R^2$. 
The set of all such planes is the coadjoint orbit. 

For the critical value $0 \in \oo(3)^* \cong \R^3$ the sphere shrinks to
the point~$G \cdot 0$.
The group orbit $G \cdot p \subseteq T^*\R^3$ turns into the sphere
$V_{3,1} = S^2$ encoding the common direction of $\xx \parallel \yy$.
For $(\xx, \yy) = 0$ this sphere shrinks to a point as well.

In higher dimension the regular situation is similar, but more complicated
by the fact that there are additional relations between the angular momenta
that define the coadjoint orbit.
These syzygies, together with $J = \const$ define the coadjoint
orbit $G \cdot \mu \cong G_{m,2}$ as a submanifold of momentum space.
E.g.\ for $m=4$ the single additional syzygy is the Pl\"ucker
relation $L_{12} L_{34} + L_{14}L_{23} + L_{13} L_{24} = 0$.
Such identities hold for any 4 indices when $m \geq 5$, but all
such identities with terms $L_{ij} L_{kl}$, $i<j, k<l$ 
are not independent of each other on their common zero level set.
As before $G_{m,2}$ is the space of all $2$--dimensional oriented planes
in $\R^m$.
No dynamics takes place in this space.
When $J>0$ the dynamics in the invariant plane is described by $(\xi,\eta)$
and the angle $\psi$ conjugate to $J$.

\section{Reduction of the degenerate Neumann System}
\label{reduce}

The results of the last section are now applied to the action of a direct 
product of orthogonal groups.
Many statements just go through because they hold separately for each
factor of the group. 
The essential difference is that we are not starting from the symplectic
manifold $T^*\R^n$ with the standard symplectic structure $[ .. \, , .. ]$,
but from the Poisson manifold $T^*\R^{n+1}$ with the Poisson structure
$\{ .. \, , .. \}$ that has $T^*S^n$ among its symplectic leaves.
The most degenerate case of the Neumann system appears when all spring
constants in the potential are equal.
This is the case treated first.

\subsection{$O(n+1)$ symmetry on $S^{n}$}
\label{sec:GeoSn}

When all coefficients in the Neumann system are equal the potential
is constant because it is proportional to the Casimir $C_1$ of the
Poisson structure \eqref{eqn:Dirac}.
Hence, the Hamiltonian describes the geodesic flow on $S^n$ with the
induced metric from $\R^{n+1}$.

Surprisingly little changes as compared to
section~\ref{orthogonalgroupactions} when 
systems on $S^n$ embedded in $\R^{n+1}$ are considered.
The reason is that the constraining Casimirs are both invariants of the
group action, in fact they are given by $C_1 = 2 V$ and $C_2 = S$.

Actually things are much nicer for this action on~$T^*S^n$, since it is
``almost'' free.
The reason is that the origin $\xx=0$ is no longer part of the configuration
space.
Moreover, since $\xx \perp \yy$ on $T^*S^n \subseteq T^*\R^{n+1}$
the only way to achieve $\xx \parallel \yy$ is to have $\yy = 0$,
i.e.\ no momentum at all.
If we consider sufficiently high energy (or geodesic flows with nonzero
energy from the start) this is impossible. 
The following theorem is well known, since it describes the geodesic
flow on the sphere; we have merely formulated it in the general framework
developed here for the degenerate Neumann system.

\begin{theorem} {\em (Dirac Bracket and $S^n$)} \label{thm:DiracSn}
   Endowing $T^*\R^{n+1}$ with the Dirac bracket \eqref{eqn:Dirac}
   the additional Casimirs $V = \frac{1}{2} C_1$ and $S = C_2$ are
   among the three basic invariants $V$, $T$, $S$ and the reduced
   brackets vanish identically.
   The reduced space merely is the point $H = T = J^2/2$. 
   The only dynamics in phase space is that generated by $J$, it is
   motion along the great circle defined by the intersection of $S^n$
   and the plane spanned by $\xx$ and $\yy$.
   The quotient of the energy surface $\{ (\xx, \yy) \in T^*S^n: H= h > 0\}$ 
   by the flow of $H$  is $J^{-1}(\sqrt{2h})/S^1 = G_{n+1,2}$.
\end{theorem}

\begin{proof}
Brackets involving Casimirs are always zero, and $\{ T, T \} = 0$ as well.
When $H$ is fixed then $J$ is also fixed and if $H > 0$ then $J$ assumes
regular values.
So the meadow for the flowers is just a single point when we fix the energy.
For a geodesic flow a change of the energy merely re-parametrises the orbits.
The petals of the flower are the one-dimensional tori parametrised by
the angle conjugate to $J$.
The symplectic leaves are given by the quotient of the energy surface
by the action of $J$. 
For $H = 1/2$ the energy surface is the unit sphere bundle
\begin{displaymath}
   T_1^*S^n = \left\{ (\xx,\yy) \in T^*\R^{n+1} :
   ||\xx|| = 1, \scap{\xx}{\yy} = 0, ||\yy|| = 1 \right\} \,.
\end{displaymath}
The flow generated by $J$ with respect to the Dirac bracket is a rotation
in the plane spanned by $\xx$ and $\yy$, as before (the original symplectic
bracket $[J, C_i]$ with the Casimirs vanishes).
Hence, the orbit is an (oriented) great circle.
The set of all oriented great circles is the Gra{\ss}mann manifold
$G_{n+1,2}$, the set of oriented $2$--planes in~$\R^{n+1}$, whence 
\begin{displaymath}
   T_1^*S^n / S^1 = G_{n+1,2} \,.
\end{displaymath}
The reversing symmetry $(\xx, \yy) \mapsto (\xx, -\yy)$ gives another orbit.
These are two orbits corresponding to each $2$--plane, 
one for each orientation.
\end{proof}

\noindent
The flow of $H$ is the flow of $J$, up to scaling time by a constant.
This is why the above result can also be obtained from
Lemma~\ref{lem:Grassmann}.

\subsection{The general degenerate case}
\label{thegeneraldegeneratecase}

Our main result is concerned with the more complicated case in which neither
all coefficients of the potential are the same, nor all of them are different. 
For each group of $m_{\sigma}$ equal coefficients there is an $O(m_{\sigma})$
symmetry group acting on the space spanned by the axes of the equal
coefficients.
In fact the action is an $O(m_{\sigma})$--action on $\R^{m_{\sigma}}$, as
described in Section~\ref{orthogonalgroupactions}.
Since the axes of the groups of equal coefficients are different, the
group action of the direct product of the $O(m_{\sigma})$ is simply
given by the direct product of their actions.
The difference to Section~\ref{orthogonalgroupactions} (besides the fact
that now we have a direct product of orthogonal groups) is that we need
to consider the Dirac bracket $\{ ..\, , .. \}$ on $\R^{n+1}$ instead of
the standard symplectic structure $[..\, , ..]$.

The reduction with respect to the action of the joint symmetry group 
is regular when the momentum of each component is nonzero.
The total angular momentum for each group of equal coefficients is 
an $\Ad^*$--invariant function.
Such functions become Casimirs after dividing out the group.
In addition there is a relation between these Casimirs that comes
from the Casimirs $C_1$ and~$C_2$ resulting from the original
embedding of the sphere in $\R^{n+1}$.
Hence, we get a reduced system that has $\ell$ degrees of freedom:
one for each group of equal coefficients minus one for the constraint
to be on the sphere.
The reduction is regular when all the fixed momenta are nonzero.
In this case the reduction leads to a simple system with an effective
potential. 
When some or all of the momenta are zero singular reduction needs
to be used.
In this case the reduced system is embedded in a higher dimensional
Euclidean space, and the reduced phase space is no longer a smooth
manifold, though still a semi-algebraic variety.

\begin{proposition} {\em (Symmetry Group Action)}
   The Neumann system with $\ell + 1$ groups of equal coefficients with 
   multiplicity $m_{\sigma} \geq 1$, $\sigma = 0, \dots, \ell$ is invariant
   under the symmetry group
   \begin{displaymath}
      G = O(m_0) \times O(m_1) \times \dots \times O(m_{\ell})
   \end{displaymath}
   with action $\Phi$ on $T^*\R^{n+1}$, $\sum m_{\sigma} = n+1$,
   \begin{displaymath}
   \begin{array}{cccc}
     \Phi : & G \times T^*\R^{n+1} &
     \longrightarrow & T^*\R^{n+1} \\
     & ( g, (\xx,\yy) ) & \mapsto & (g\xx, g\yy) \,.
   \end{array}
   \end{displaymath}
   Here $g$ is a block diagonal matrix from $G \subseteq O(n+1)$.
\end{proposition}

\begin{proof}
This is clear from the structure of the degenerate Neumann system. 
Each group of $m_{\sigma}$ equal coefficients $b_{\sigma} = a_i$,
$i \in I_{\sigma}$ admits an $O(m_{\sigma})$ symmetry acting on
the co-ordinates $x_i$ with $i \in I_{\sigma}$ and the 
corresponding momenta $y_i$.
The sets $I_\sigma$ are disjoint, hence the action is a direct product. 
\end{proof}

\noindent
When the potential of the Neumann system is written in
general form as $\frac{1}{2} \scap{\xx}{{\bf A} \xx}$ we see that the symmetry
group is in fact the group that describes the ambiguity of the choice
of a diagonalizing co-ordinate system when ${\bf A}$ has multiple
eigenvalues.

\begin{theorem} {\em (Momentum Mapping)}
   Denote by $\widetilde{G}$ the subgroup of $G$ containing
   the factors~$O(m_{\sigma})$ with $m_{\sigma} \geq 2$ and denote the number
   of such factors by $\tilde\ell + 1$.
   The Momentum mapping $\J$ of the $\Phi$--action of
   $\widetilde{G} \subset G$ is the direct product of the
   Momentum mappings of the individual factors for which $m_{\sigma} \geq 2$.
   It is $\Ad^*$--equivariant with momentum
   $\mu \in \widetilde{\gg}$.
   The isotropy groups are the direct products of the isotropy groups
   of Lemma~\ref{lem:momOm}.
   To each partial mapping corresponding to indices $I_{\sigma}$ with
   $m_{\sigma} \geq 2$ there is a Casimir $W_{\sigma}$ given by the
   $\Ad^*$--invariant function
   \begin{displaymath}
      W_{\sigma}(\xx, \yy) = \sum_{i<k \in I_{\sigma}} L_{ik}^2(\xx,\yy) \,.
   \end{displaymath}
\end{theorem}

\begin{proof}
Since each partial action of $O(m_{\sigma})$ is as described in
Lemma~\ref{lem:momOm} this result immediately follows.
When $m_{\sigma} = 1$ the corresponding group only has a discrete factor
$O(1)$, which acts by reflection in configuration space leading to the
$\pi$--rotation $(x_{\sigma}, y_{\sigma}) \mapsto (-x_{\sigma}, -y_{\sigma})$. 
This discrete symmetry does not appear in the Momentum mapping. 
Hence, the momentum $\mu$ only contains the components for 
each group with $m_{\sigma} \geq 2$.
\end{proof}

\begin{remark}  \label{rem:signed}
When $m_\sigma = 2$ it is also useful to consider $SO(2)$ reduction
instead of the $O(2)$ reduction.
The reason is that then $L_{ik}$, $i < k \in I_\sigma$ is among the
invariants and $W_{\sigma} = L_{ik}^2$ can be replaced by
$J_{\sigma} = L_{ik}$.
In this way the action $J_{\sigma}$ becomes a signed quantity that is
differentiable also at $||\mu_\sigma|| = 0$, because no square root
needs to be extracted. 
This is particularly important at $||\mu_\sigma|| = 0$ because then the
fixed set of the discrete symmetry is accessible for the dynamics.
It is not possible to define such signed~$J_\sigma$ when $m_\sigma \ge 3$
because then the plane of rotation cannot be defined when
$||\mu_\sigma || = 0$, while for $m_\sigma = 2$ there only is one plane.
\end{remark}

\noindent
Denote by $\Pro_{\sigma}$ the projection onto the co-ordinates with
indices $I_{\sigma}$, so that $\Pro_{\sigma} \xx$ gives all the
co-ordinates $x_i$ with $i \in I_{\sigma}$, similarly
$\Pro_{\sigma} \yy$ gives the corresponding momenta.
In each such subspace we have the standard
Eulidean scalar product and norm.

\begin{theorem} {\em (Regular Reduction)} \label{thm:RegRed}
   Let each component of $\mu \in \widetilde{\gg}$ be nonzero, hence
   the Casimir $W_{\sigma} \neq 0$ for each $m_{\sigma} \geq 2$.
   For each such $\sigma$ define reduced co-ordinates by
   \begin{displaymath}
      \xi_{\sigma} = || \Pro_{\sigma} \xx ||, \quad
      \eta_{\sigma} = \frac{ \scap{\Pro_{\sigma} \xx}{\Pro_{\sigma} \yy} }
                           {|| \Pro_{\sigma} \xx||} \,.
   \end{displaymath}
   When $m_{\sigma} = 1$ define $\xi_{\sigma} = x_{\sigma}$ and
   $\eta_{\sigma} = y_{\sigma}$.
   The reduction  mapping $R$ from $(\xx, \yy)$ to  $(\xib, \etab)$ is a
   Poisson mapping, i.e.\ it preserves the Dirac bracket $\{ .. \, , .. \}$.
   The $G$--reduced phase space is an open subset of $T^*S^{\ell}$.
   The reduced Hamiltonian reads
   \begin{equation} \label{eqn:HamRed}
      H =  \frac{1}{2} \sum_{\sigma=0}^{\ell} \eta_{\sigma}^2 + V_{\mu}, \qquad
      V_{\mu} = \frac{1}{2} \sum_{\sigma = 0}^{\ell} b_{\sigma} \xi_{\sigma}^2
      + \frac{W_{\sigma}}{ \xi_{\sigma}^2} \,.
   \end{equation}
   where formally $W_{\sigma} \equiv 0$ when $m_{\sigma} = 1$.
\end{theorem}

\begin{proof}
The verification that this is a Poisson mapping can be done by direct
computation.
The reduction is as described in Lemma~\ref{lem:RegRedOm}, now applied
separately to each group of indices $I_{\sigma}$ with $m_{\sigma} \geq 2$.
The fact that the reduction does not lead to all of $S^{\ell}$ comes from
the fact that $\xi_{\sigma} = 0$ is impossible when
$||\mu_{\sigma}|| \neq 0$.
Nevertheless, the global Casimirs $C_1 = \sum \xi_{\sigma}^2 = 1$ 
and $C_2 = \sum \xi_\sigma \eta_\sigma = 0$
restrict the motion to the reduced phase space~$T^*S^\ell$.
\end{proof}

\noindent
Clearly, the reduced Hamiltonian \eqref{eqn:HamRed} defines an
integrable system on $S^{\ell}$.
This system was first studied by Rosochatius~\cite{rosochatius77}
who separated it in elliptical-spherical co-ordinates.

\begin{proposition} {\em (Generalized action-angle co-ordinates)}
   The Poisson structure~\eqref{eqn:Dirac} on $T^*\R^{n+1}$ locally
   near a weakly regular value of the Energy--Momentum
   mapping~$(H, \J)$ has the bracket relations
   \begin{displaymath}
   \begin{array}{rclcl}
      \{ \xi_{\sigma}, \eta_{\tau} \} & = &
      \delta_{\sigma \tau} - \frac{\xi_{\nu} \xi_{\tau}}{C_1} \, ,
      & \qquad & \sigma, \tau = 0, \ldots, \ell  \\
      \{ \eta_{\sigma}, \eta_{\tau} \} & = &
      \frac{\xi_\tau \eta_\sigma - \xi_\sigma \eta_\tau}{C_1} \,,
      & \qquad & \sigma, \tau = 0, \ldots, \ell  \\
      \{ \psi_{\varsigma}, J_{\tau} \} & = & \delta_{\varsigma \tau} \, ,
      & \qquad & \varsigma, \tau = 0, \ldots, \tilde{\ell}  \\
      \{  L_{ik}, L_{kl} \} & = & L_{li} \, ,
      & \qquad & i, k, l \in I_{\sigma} \, , \, \sigma = 0, \ldots, \ell
   \end{array}
   \end{displaymath}
   with all further bracket relations being zero. 
   The reduced Hamiltonian \eqref{eqn:hatH} only depends on the
   co-ordinates $(\xi, \eta, J)$.
   A system of generalized action-angle co-ordinates in the sense
   of Nekhoroshev~\cite{nekhoroshev72} is obtained by action-angle co-ordinates
   $(\phi_1, \ldots, \phi_{\ell}, I_1, \ldots, I_{\ell})$
   of the reduced system and by
   $(\psi_0, \ldots, \psi_{\tilde{\ell}}, J_0, \ldots, J_{\tilde{\ell}})$
   together with any sets of $\sum 2 m_{\sigma} - 4$ symplectic
   co-ordinates on the Gra{\ss}mannians $G_{m_{\sigma},2}$
   with $m_{\sigma} \geq 3$.
\end{proposition}

\begin{proof}
The proof of this is similar to that of Lemma~\ref{lem:nekho}. 
The difference is that here we work in the Poisson structure instead of 
in the symplectic structure.
While this saves us the work of specifying a local symplectic structure
on~$T^*S^\ell$, the number $\frac{1}{2} m_{\sigma} (m_{\sigma} - 1)$ of
global co-ordinates~$L_{ik}$ exceeds the dimension $2 m_{\sigma} - 4$
of~$G_{m_{\sigma},2}$ and the syzygy
\begin{displaymath}
   J_{\sigma}^2 = \sum_{i < k \in I_{\sigma}} L_{ik}^2
\end{displaymath}
for $m_{\sigma} \geq 4$ is accompanied by increasingly more syzygies.
The fact that $J_{\sigma}$ has periodic flow with respect to the
Dirac bracket can be proved as in Theorem~\ref{thm:DiracSn}.
\end{proof}

\noindent
By the separation of variables in the next section  we give a local
symplectic co-ordinate system on $T^*S^\ell$ and are then able to compute
$I_1, \ldots, I_{\ell}$ explicitly.

\begin{theorem} {\em (Energy--Casimir mapping, regular part)}
   The fibre over each weakly regular point of the Energy--Casimir mapping 
   \begin{displaymath}
   \begin{array}{cccc}
      \EC : & T^*\R^{n+1} & \longrightarrow & \R^{\tilde\ell+2} \\
      & (\xx,\yy) & \mapsto &  (H, J_0, \dots, J_{\tilde\ell})
   \end{array}
   \end{displaymath}
   is given by
   \begin{displaymath}
     {\cal F}_{h,j} = \EC^{-1} ( h,j ) = H^{-1}(h) \cap 
          \bigcup_{\tilde \mu \in G \cdot \mu} \J^{-1}(\tilde \mu) \,.
   \end{displaymath}
   The structure of ${\cal F}_{h,j}$ is that of a double bundle in view of the 
   two properties
   \begin{displaymath}
          \J( {\cal F}_{h,j} ) = \prod_{\sigma = 0}^{\ell} G_{m_{\sigma}, 2},
          \qquad
          R(  {\cal F}_{h,j} ) =  H^{-1}(h) \cap J^{-1}(j) \cap R(T^*S^n)
          = N_{h,j}^{\ell}
   \end{displaymath}
   where $R$ is the reduction mapping of Theorem~\ref{thm:RegRed}.
   The fibre of each of the two mappings is a $\T^{\tilde{\ell} + 1}$
   bundle over the image of the other one, yielding the diagram

   \begin{picture}(120,130)
      \put(13,20){$N_{h,j}^{\ell}$}
      \put(58,90){\vector(-1,-2){25}}
      \put(55,100){${\cal F}_{h,j}$}
      \put(73,90){\vector(1,-2){25}}
      \put(83,20){$\prod G_{m_{\sigma}, 2}$.}
      \put(48,60){\vector(2,-1){43}}
      \put(33,60){$R$}
      \put(83,60){\vector(-2,-1){43}}
      \put(88,60){$\J$}
      \put(55,60){$\T^{\tilde{\ell} + 1}$}
   \end{picture}

   \noindent
   Relative equilibria are given by the critical points of the amended
   potential $V_{\mu}$: $\xi_{\sigma} = 0$ for $m_{\sigma} = 1$ and
   $\xi_{\sigma}^2 = j_{\sigma}/\omega_\sigma$,
   $\omega_\sigma = \sqrt{b_{\sigma}-\beta}$ otherwise,
   where $\beta$ is determined from $\sum \xi_{\sigma}^2 = 1$.
   The corresponding critical value of the Energy--Casimir mapping is 
   \begin{displaymath}
      (h,j)  =  \left( \sum j_{\sigma}
      (\omega_{\sigma} + b_\sigma / \omega_{\sigma}), j \right) \,.
   \end{displaymath}
   There are no other critical values of~$\EC$ in the range
   \begin{displaymath}
      \left\{ (h, j) \in \R^{\tilde{\ell} + 1} : j_{\sigma} > 0
      \, \forall_{\varsigma = 0, \ldots, \tilde{\ell}}  \right\}
   \end{displaymath}
   of regular reduction.
\end{theorem}

\begin{proof}
The double bundle structure appears because we chose to treat the actions $J$
that come from the group $\widetilde{G}$ apart from the actions $I$ that
appear in the reduced system.
In the end we obtain a $\T^{\ell + \tilde \ell + 1}$--bundle, but for now we
only understand the actions $J$ because they are related to the Momentum
mapping of~$\widetilde{G}$.
In particular for the relative equilibria (relative to $\widetilde{G}$!) this
splitting is natural.

Relative equilibria are solutions of
$\nabla H =  \frac12 \beta \nabla C_1 + \delta \nabla C_2$.
Since $C_2 = 0$ the last $n$ components of these equations give
$\etab = 0$ and $\delta = 0$.
Hence $\nabla_\xi H =  \beta \xib$ needs to be satisfied. 
This leads to
$ \xi_{\sigma} b_{\sigma} -  j_{\sigma}^2 / \xi_{\sigma}^3
 -  \beta \xi_{\sigma} = 0$, 
which gives the result.
The multiplier $\beta$ is determined by the constraint $C_1  = 1$.
\end{proof}

\begin{lemma}
The critical energy $h$ of the relative equilibria is a convex function of the momenta $j_\sigma$.
\end{lemma}

\begin{proof}
The function $h(j)$ is implicitly defined by $h = f(j, \beta)$ and
$g(h, j, \beta) = 0$, where $f$ is given by the Energy--Casimir mapping
and $g$ is given by the constraint.
We are going to prove that the Hessian of $h(j)$ is positive semi-definite.
By implicit differentiation we find
\[
\partial_{j_\sigma} h = \partial_{j_\sigma} f - \partial_{\beta} f
\frac{\partial_{j_\sigma} g}{\partial_{\beta} g}
= \omega_\sigma + b_\sigma / \omega_\sigma - \beta/ \omega_\sigma
= 2 \sqrt{b_\sigma - \beta} = 2 \omega_\sigma
\] 
where the identity $\partial_{\beta} f / \partial_{\beta} g = \beta$
has been used.
To compute the second derivative
we therefore need the derivative of~$\beta$.
It is given by 
\[
\partial_{j_\sigma} \beta = - \frac{\partial_{j_\sigma} g}{\partial_{\beta} g}
= -\frac{2}{\omega_\sigma}
\left( \sum_{\nu=0}^\ell \frac{ j_\nu }{ \omega_\nu^3} \right)^{-1}  \,,
\]
Combining the two first derivatives gives the entries of the Hessian as
\[
\frac{\partial^2 h}{\partial j_\sigma \partial j_\tau} =
\frac{2}{\omega_\sigma \omega_\tau}
\left( \sum_{\nu=0}^\ell \frac{ j_\nu }{ \omega_\nu^3} \right)^{-1}  \,.
\]
This is a rank one matrix with one positive eigenvalue, and hence $h(j)$ is
a convex function.
\end{proof}

\noindent
In order to describe the invariant manifolds of the invariant sets contained
in lower dimensional invariant Neumann systems defined by $j_{\sigma} = 0$
for some (or all) $\sigma$ singular reduction is used.

\begin{theorem} {\em (Singular Reduction)} \label{thm:SingRed}
   The Hilbert mapping $\chi$ of the $G$--action is the direct
   product of the Hilbert mappings of the individual factors $O(m_{\sigma})$,
   as given in Lemma~\ref{lem:SingRedOm}. 
   The reduction mapping $\chi : T^*\R^{n+1} \longrightarrow \R^{3(\ell+1)}$
   is given by
   \begin{displaymath}
      (\scap{\Pro_{\sigma} \xx}{\Pro_{\sigma} \xx},
       \scap{\Pro_{\sigma} \yy}{\Pro_{\sigma} \yy},
       \scap{\Pro_{\sigma} \xx}{\Pro_{\sigma} \yy})
      = (2V_{\sigma}, 2T_{\sigma}, S_{\sigma})
   \end{displaymath}
   with syzygies $2 V_{\sigma} T_{\sigma} = \frac{1}{2} S_{\sigma}^2$
   whenever $m_{\sigma} = 1$.
   The reduced brackets are 
   \begin{displaymath}
   \begin{aligned}
      \{ V_{\sigma}, T_{\tau} \} & =
      S_{\tau}(\delta_{\sigma \tau} - \frac{2 V_{\sigma}}{C_1}), &
      \{ V_{\sigma}, S_{\tau} \} & =
      2 V_{\sigma}( \delta_{\sigma \tau} - \frac{2 V_{\tau}}{C_1}), \\
      \{ T_{\sigma}, S_{\tau} \} & =
      - 2 T_{\sigma} ( \delta_{\sigma \tau} - \frac{2 V_{\tau}}{C_1}), \quad &
      \{ T_{\sigma}, T_{\tau} \} & =
      \frac{2 T_{\sigma} S_{\tau} - 2 T_{\tau} S_{\sigma}}{C_1}, \\
      \{ S_{\sigma}, S_{\tau} \} & = 0, &
      \{ V_{\sigma}, V_{\tau} \} & = 0 \,.
   \end{aligned}
   \end{displaymath}
   This bracket has rank $2 \ell$ and the $\tilde{\ell} + 3$ Casimirs
   \begin{displaymath}
      C_1 = 2 \sum V_{\sigma}, \quad
      C_2 = \sum S_{\sigma}  \quad \mbox{and} \quad
      W_{\sigma} = 4 V_{\sigma} T_{\sigma} - S_{\sigma}^2 \,.
   \end{displaymath}
   From the latter we obtain at nonzero values the $\tilde{\ell} + 1$
   actions $J_{\sigma} = \sqrt{W_{\sigma}}$.
   The reduced Hamiltonian of the degenerate Neumann system reads
   \begin{equation}  \label{eqn:redegNeu}
      \hat{H}(V, T, S) = \sum_{\sigma = 0}^{\ell} T_{\sigma}
      + b_{\sigma} V_{\sigma} \,.
   \end{equation}
\end{theorem}

\begin{proof}
The brackets with equal indices on the left hand side are found by direct
calculation using \eqref{eqn:genDirac}.
Only the additional terms need to be computed, for the $[ .. \, , .. ]$
brackets see Lemma~\ref{lem:SingRedOm}.
Since $C_1 = 2V$ and $C_2 = S$ the additional terms are also already known.
The other brackets follow similarly by direct computation.
\end{proof}

\begin{remark} \label{rem:SingRed}
For the co-ordinates with $m_{\sigma} = 1$ we may as well keep the
original $x_{\sigma}, y_{\sigma}$ and avoid the $\ell - \tilde{\ell}$
syzygies.
Note that the rank of the bracket remains $2 \ell$ with $\tilde{\ell} + 3$
Casimirs.
All $S_{\sigma}$ can be eliminated using the syzygies and fixing the
values $w_{\sigma}$ of the Casimirs~$W_{\sigma}$, albeit at the
cost of extracting a square root.
The two sheets of $\pm S_{\sigma}$ are glued together smoothly at
$S_{\sigma} = 0$.
In addition using $C_1, C_2$ this projection of the reduced space
is of dimension $2\ell$, the same as the dimension of the regularly
reduced phase space.
\end{remark}

\section{The Rosochatius system}
\label{rosochatius}

Regular reduction of the degenerate Neumann system leads 
 to an open subset of~$T^*S^{\ell}$ as described in Theorem~\ref{thm:RegRed}.
The reduced phase spaces $\J^{-1}(\mu)/G_{\mu}$ can be interpreted
as the regular symplectic leaves of~$T^*\R^{n+1}/G$, which we
embedded in~$\R^{3(\ell + 3)}$ in Theorem~\ref{thm:SingRed}.
The remaining symplectic leaves of $T^*\R^{n+1}/G$ can be given
an interpretation of ``filling up the remaining part''
of~$T^*S^{\ell}$, as already indicated in Remark~\ref{rem:SingRed}.

Theorem~\ref{thm:RegRed} suggests to consider the following ``unfolding''
of the reduced Neumann system on $T^*S^{\ell} \subseteq T^*\R^{\ell + 1}$ with
co-ordinates $(\xi, \eta)$ satisfying the Dirac brackets~\eqref{eqn:Dirac}
in dimension $2\ell + 2$.
The Hamiltonian is given by
\begin{equation}  \label{eqn:rosochatius}
   H(\xi, \eta)  =  \frac{1}{2} \sum_{\sigma = 0}^{\ell}
   \eta_{\sigma}^2 + b_{\sigma} \xi_{\sigma}^2 +
   \frac{w_{\sigma}}{\xi_{\sigma}^2} \,,
\end{equation}
where $w_{\sigma}$, $\sigma = 0, \ldots, \ell$ are the unfolding parameters.
When reducing the degenerate Neumann system we find
$w_{\sigma} = j_{\sigma}^2 \geq 0$, and it is this same condition that
is used in~\cite{rosochatius77, Moser80b, Macfarlane92} to define the
Rosochatius system on~$T^*S^{\ell}$.
Since $b_0 < \ldots < b_\ell$ the symmetry group
$\Z_2^{\ell +  1} = O(1) \times \ldots \times O(1)$ is discrete.

\begin{theorem}  \label{thm:Rosochatius}
   Singular reduction of the $\Z_2^{\ell +  1}$--symmetry of the
   Rosochatius system leads to the same reduced dynamics as the
   singularly reduced degenerate Neumann system.
   The Hilbert mapping 
   $\chi : T^*\R^{\ell+1} \longrightarrow \R^{3(\ell+1)}$
   is given by the basic invariants
   \begin{displaymath}
      V_{\sigma} = \frac{\xi_{\sigma}^2}{2} , \quad
      T_{\sigma} = \frac{\eta_{\sigma}^2}{2}
      + \frac{w_{\sigma}}{2 \xi_{\sigma}^2} ,
      \quad \mbox{and} \quad
      S_{\sigma} = \xi_{\sigma} \eta_{\sigma}
   \end{displaymath}
   which satisfy the Poisson bracket relations of Theorem~\ref{thm:SingRed}
   and turn the Hamiltonian~\eqref{eqn:redegNeu} into~\eqref{eqn:rosochatius}.
\end{theorem}

\begin{proof}
This follows again by direct computation, e.g.\
\begin{eqnarray*}
   \{ T_{\sigma}, T_{\tau} \} & = & \eta_{\sigma} \eta_{\tau}
   \frac{\xi_{\tau} \eta_{\sigma} - \xi_{\sigma} \eta_{\tau}}{C_1} +
   \eta_{\sigma} \frac{w_{\tau}}{\xi_{\tau}^3}
   (\delta_{\sigma \tau} - \frac{\xi_{\sigma} \xi_\tau}{C_1}) -
    \frac{w_{\sigma}}{\xi_{\sigma}^3} \eta_{\tau}
   (\delta_{\sigma \tau} - \frac{\xi_{\sigma} \xi_\tau}{C_1})  \\ & = &
   \frac{2 T_{\sigma} \cdot S_{\tau} - 2 T_{\tau} \cdot S_{\sigma}}{C_1} \,.
\end{eqnarray*}
\end{proof}

\noindent
In this way the Rosochatius system is a $2^{\ell}$--fold
covering of the (singularly) reduced degenerate Neumann system.

\begin{remark}  \label{rem:Rosochatius}
For coefficients $b_{\sigma}$ of the degenerate Neumann system
with $m_{\sigma} = 1$ we may keep the original
$x_{\sigma}, y_{\sigma}$ variables~; for these indices
$w_{\sigma} = 0$ and the Rosochatius system is a
$2^{\tilde{\ell}+1}$--fold covering of the resulting reduced system
with $\widetilde{G}$ reduced instead of~$G$.
For coefficients $b_{\sigma}$ of the degenerate Neumann system
with $m_{\sigma} = 2$ we may similarly choose to reduce with respect
to the $SO(2)$--factor instead of the $O(2)$--factor.
In particular, if $m_0 = \ldots = m_{\ell} = 2$ then singular
reduction of the degenerate Neumann system with respect to the
subgroup $SO(2) \times \ldots \times SO(2)$ of
$G = O(2) \times \ldots \times O(2)$ yields the (unreduced)
Rosochatius system, as had already been remarked
in~\cite{Moser80b, Macfarlane92}.
\end{remark}

\subsection{Separation}
\label{separate}

In this section we consider the reduced system in its own right
and separate the Hamiltonian of the Rosochatius system in
elliptical-spherical co-ordiates.
Note that for the integrability the condition $w_{\sigma} \geq 0$
is not needed.
However, the dynamics with $w_{\sigma} < 0$ is completely different
as then there are solutions that blow up in finite time.

\begin{theorem}  \label{thm:degenerateseparation}
   The Neumann system on $S^{\ell}$ with additional potential
   $\frac{1}{2} \sum w_{\sigma}/\xi_{\sigma}^2$ can be separated
   in elliptical-spherical co-ordinates $u_i$. 
   The general solution of the Hamilton--Jacobi equations is
   \begin{displaymath}
      S =\frac{1}{2} \sum_{i=1}^{\ell} \int^{u_i} \frac{\zeta}{A(z)} \dee z
   \end{displaymath}
   where $A(z) = \prod (z - b_{\sigma})$ and the integral is defined on 
   the real part of the hyperelliptic curve 
   \begin{displaymath}
      \Gamma = \{ (z,\zeta) \in \C^2 : \zeta^2 =
      -  Q(\rho; z)A(z) + \tilde Q(b, w ; z)  \}
   \end{displaymath}
   of genus $\ell$.
   The separation constants are $\rho_i$ in
   \begin{displaymath}
     Q(\rho; z) = z^{\ell} + 2\rho_1 z^{\ell - 1} + \dots + 2 \rho_{\ell}
   \end{displaymath}
   where $\rho_1$ is the value of the energy.
   The polynomial $\tilde{Q}$ is given by 
   \begin{displaymath}
      \tilde{Q}(b, w ; z)  =
      \sum_{\sigma=0}^{\ell} w_{\sigma} \prod_{\tau \neq \sigma}^{\ell}
      (z-b_{\tau})(b_{\sigma}-b_{\tau})\,.
   \end{displaymath}
\end{theorem}

\begin{proof}
The separation proceeds as in the case without additional potential,
see~\cite{Moser80}.
Starting from the Lagrangian in spherical co-ordinates defined by 
\begin{displaymath}
   f(z) = \sum_{\sigma = 0}^{\ell} \frac{ \xi_{\sigma}^2}{z-b_{\sigma}} = 0 \,,
\end{displaymath}
the co-ordinates are the roots $u_i$, $i = 1,\dots, \ell$ of the rational
function $f(z) = 0$.
Because the $b_{\sigma}$ are different, and $f(z)$ has a pole of first
order at each $b_{\sigma}$ the graph of $f(z)$ shows that the roots
$u_{\sigma}$ satisfy
\begin{displaymath}
   b_0 \leq u_1 \leq b_1 \leq \dots \leq u_{\ell} \leq b_{\ell} \,.
\end{displaymath}
The denominator of the rational function $f(z)$ is the polynomial $A(z)$.
Define the polynomial in the numerator of $f(z)$ as $U(z) = \prod (z-u_i)$.
Then we have
\begin{displaymath}
   \xi_{\sigma}^2 = \res_{z=b_{\sigma}} f(z) \dee z
   = \frac{U(b_{\sigma})}{A'(b_{\sigma})} \,.
\end{displaymath}
From the standard approach, see~\cite{dullin01} for the details, we 
know that after introducing the momentum $p_i$ conjugate to $u_i$ 
the kinetic and original Neumann-potential energy read
\begin{displaymath}
\begin{aligned}
   T(u,p) &= \frac{1}{2} \sum_{i=1}^{\ell} \frac{p_i^2}{g_i(u)}
   \quad \mbox{with} \quad
   \frac{1}{g_i(u)} = -4 \frac{A(u_i)}{U'(u_i)}, \\
   V(u) &= \frac{1}{2} \sum_{\sigma = 0}^{\ell} b_{\sigma}
   - \frac{1}{2} \sum_{i=1}^{\ell} u_i\,. 
\end{aligned}
\end{displaymath}
The additional potential $V_w$ in terms of $u$ becomes
\begin{displaymath}
   V_w(u) = \frac{1}{2} \sum_{\sigma = 0}^{\ell} w_{\sigma}
   \frac{A'(b_{\sigma})}{U(b_{\sigma})}
\end{displaymath}
so that
\begin{equation}
\label{thehamiltonianinuandp}
   H(u,p) = T(u,p) + V(u) + V_w(u)
\end{equation}
and the Hamilton--Jacobi equation reads
$H(u, \partial S/ \partial u ) = \rho_1$.
The general solution for $S$ depends on $\ell$ integration constants.
We write the Hamilton--Jacobi equation in the form 
$0 = \sum h(u_i, p_i) / U'(u_i)$ where $h(u_i, p_i)$
only depends on $u_i$ and $p_i$.
The kinetic energy $T(u,p)$ already has this form.
To achieve this form for $V(u)$ and $ \rho_1$ the following two
identities due to Jacobi are used, see Appendix~\ref{jacobitrick}.
For any polynomial $U$ of degree $\ell$ with roots $u_i$ we have
\begin{displaymath}
   \sum_{i=1}^{\ell} u_i = \sum_{i=1}^{\ell} \frac{u_i^{\ell}}{U'(u_i)} \,.
\end{displaymath}
If in addition we have a polynomial
$P(z) = \rho_1 z^{\ell - 1} + \dots + \rho_{\ell}$ with 
arbitrary coefficients then
\begin{displaymath}
  \rho_1 = \sum_{i=1}^{\ell} \frac{ P(u_i) } {U'(u_i)} \,.
\end{displaymath}
It remains to write $V_w(u)$ in the desired form.
Consider the partial fraction decomposition of $\xi_{\sigma}^{-2}$ with 
respect to $b_\sigma$, 
\begin{displaymath}
   \frac{1}{\xi_{\sigma}^2} = \frac{A'(b_{\sigma})}{U(b_{\sigma})} =
   A'(b_{\sigma}) \sum_{i=1}^{\ell} \frac{1}{U'(u_i)( b_{\sigma} - u_i)}\,.
\end{displaymath}
Now exchange the summation over $i$ in this formula and $\sigma$
in $V_w$ to find
\begin{equation} \label{Qtilde}
   V_w = \frac{1}{2}
   \sum_{i=1}^{\ell} \frac{1}{U'(u_i)} \sum_{\sigma = 0}^{\ell}
   w_{\sigma} \frac{A'(b_{\sigma})}{b_{\sigma} - u_i}
   = \sum_{i=1}^{\ell}
   \frac{1}{U'(u_i)} \frac{\tilde{Q}(b, w ; u_i)}{2A(u_i)}\,,
\end{equation}
where
\begin{displaymath}
   \tilde{Q}(b, w ; u_i) = \sum_{\sigma = 0}^{\ell}
   \frac{w_{\sigma}}{b_{\sigma} - u_i} A(u_i) A'(b_{\sigma}) \,.
\end{displaymath}
Hence, every term in~\eqref{thehamiltonianinuandp} is proportional to
$1/U'(u_i)$ (up to the constant $c  =\frac{1}{2} \sum b_{\sigma}$)
and $H$ reads
\begin{displaymath}
   H(u,p) = c +  \sum_{i=1}^{\ell} \frac{1}{U'(u_i) }
   \left( -2A(u_i) p_i^2 - \frac{1}{2} u_i^{\ell}
   + \frac{\tilde{Q}(b, w ; u_i)}{2A(u_i)} \right) \,.
\end{displaymath}
The Hamilton--Jacobi equation $H(u,\partial S/\partial u) = \rho_1$ becomes
\begin{displaymath}
   0 = \sum_{i=1}^{\ell} \frac{1 }{U'(u_i)}\, h(u_i,\partial S/\partial u_i)\,,
\end{displaymath}
where after shifting $\rho_1$ by $c$ we find
\begin{displaymath}
   h(u_i,p_i) =  -2A(u_i) p_i^2 - \frac{1}{2} u_i^{\ell} - P(u_i)
   + \frac{\tilde{Q}(b, w ; u_i)}{2A(u_i)} \,.
\end{displaymath}
To complete the proof define 
\begin{displaymath}
   Q(\rho; z) = z^{\ell} + 2 P(z) =
   z^{\ell} + 2\rho_1 z^{\ell - 1} + \dots + 2 \rho_{\ell} \,,
\end{displaymath}
and the general solution of the Hamilton--Jacobi equation depending
on $\ell$ separation constants $\rho_i$, $i = 1, \dots, \ell$ is
\begin{displaymath}
   S = \frac{1}{2} \sum_{i=1}^{\ell}
   \int^{u_i} \frac{\zeta(z)}{A(z)} \dee z \,,
\end{displaymath}
where $\zeta$ is found by solving $h(z,\zeta/(2A(z))) = 0$. 
To make the integral well-defined in the presence of the square root
it is defined on the hyperelliptic curve given by
\begin{displaymath}
    \zeta^2 = - Q(\rho; z) A(z)  + \tilde{Q}(b, w ; z) \,.
\end{displaymath}
\end{proof}

\begin{remark}  \label{rem:genus}
The genus of the curve $\Gamma$ is~$\ell$, which is unchanged by introduction 
of $w_{\sigma}$ since $QA$ has degree $2 \ell + 1$ while $\tilde{Q}$
only has degree $\ell$.
Hence for the reduced system the genus equals $\ell$, which is the number 
of different coefficients minus one.

When $\ell = 0$ the curve is algebraic, this corresponds to the Neumann system
with all axes equal, hence it is the geodesic flow on the sphere with all 
orbits closed, see Section~\ref{sec:GeoSn}.

When $\ell = 1$ the curve is elliptic and explicity given by (recall from
reduction that $w_{\sigma}$ is the value of $J_{\sigma}^2$)
\begin{displaymath}
   \zeta^2 = -(z+2\rho_1)(z-b_0)(z-b_1) + J_0^2(z-b_1)(b_0-b_1)
   + J_1^2(z-b_0)(b_1-b_0) \,.
\end{displaymath}
If $J_0 = J_1 = 0$ the Neumann system on the circle is recovered.

When $\ell = 2$ the curve is hyperelliptic, containing the classical 
Neumann system on the sphere as a special case for $J_0 = J_1 = J_2 = 0$.
\end{remark}

\begin{remark}  \label{Lagrange}
For $\ell = 1$ the curve given in the previous remark is the same as for the 
Lagrange top.
This equivalence is classical and has first been described  by Klein \&
Sommerfeld~\cite{KS}: when the spherical Lagrange top is described on 
the double covering $SU(2) \cong S^3$ of $SO(3)$ then the equations of motion 
of the Lagrange top are those of the Neumann system with potential
$b_0( x_0^2 + x_1^2) + b_1(x_2^2 + x_3^2)$, hence $\ell = 1$, $m_0 = m_1 = 2$.
As a result we can conclude that this degenerate Neumann system has Hamiltonian
monodromy, inherited from the monodromy of the Lagrange top~\cite{cushman97}.

The relation between the Lagrange top on $SO(3)$ and the degenerate
Neumann system  on $S^3$ is interesting because the action of the symmetry
group $\T^2$ in the Lagrange top is not free, although the constituting
$SO(2)$--actions are both free, while for the Neumann system each
$SO(2)$--action already has the trivial fixed points.
The resulting period lattice for the Lagrange top is generated by 
$(J_0 \pm J_1)/2$, and not by $J_0, J_1$.
This is an expression of the fact that $SO(3)$ has half the volume 
of $S^3$, which covers it twice.
\end{remark}

\subsection{Actions and Frequencies}
\label{actionsandfrequencies}

The degenerate Neumann system has up to $\ell + 1$ global actions
$J_{\sigma}$, and for simplicity during the rest of this section
we assume that this maximal number is realised, i.e.\
$\tilde{\ell} = \ell$ and all $m_\sigma \geq 2$.
In addition to the $\tilde{\ell} + 1$ actions $J_{\sigma}$ the
separation of variables of the reduced system gives $\ell$ non-trivial
actions~$I_i$.
The integration proceeds between the branch points of the hyperelliptic
curve~$\Gamma$.
Since the co-ordinates $u$ cover only one $2^{\ell+1}$--tant of the sphere
$S^\ell$ the signs of the variables $\xi$, $\eta$ need to be recovered.
This is very similar to the non-degenerate Neumann system,
see~\cite{dullin01} for the details.
Recovering the sign is crucial when some or all $J_{\sigma}$ vanish
because then the co-ordinate planes can be crossed and hence a passage
to another $2^{\ell+1}$--tant takes place.
The co-ordinate $\xi_{\sigma}$ becomes zero when  $u_i = b_{\sigma}$ for
some $i \in \{ 1, \ldots, \ell \}$.
Consequently, when one of the endpoints
of the segment $ [z_{2i-1}, z_{2i}]$ is equal to $b_{\sigma}$ then the
co-ordinate $\xi_{\sigma}$ changes sign along the cycle $c_i$ encircling
this segment of positive $\zeta$ in the complex $z$--plane. 
This means that the cycle $\gamma_i$ on the Liouville torus corresponds
to the cycle $c_i$ taken twice.
This proves the following result.

\begin{theorem}
   The non-trivial actions of the reduced degenerate Neumann system
   are given by 
   \begin{displaymath}
       I_i = \frac{1}{4\pi} \oint_{\gamma_i} \frac{\zeta}{A} \dee z 
       = \frac{1}{4\pi} \oint_{\gamma_i}
       \left( -Q + \frac{\tilde{Q} }{ A }\right) \frac{\dee z}{\zeta}
       = \frac{1}{4\pi} \oint_{\gamma_i}
       \left( -Q + \sum w_{\sigma} \frac{A'(z)}{z - b_{\sigma}} \right)
       \frac{\dee z}{\zeta} \,,
   \end{displaymath}
   where integration path $\gamma_i$ equals $c_i$ if the segment
   $[z_{2i-1}, z_{2i}]$ contains none of $b_0, \dots, b_\ell$ and
   $\gamma_i = 2 c_i$ otherwise.
   $\,$ \hfill $\Box$
\end{theorem}

\noindent
When the Neumann system is reduced by the discrete symmetry group 
of reflection of each co-ordinate then for the reduced system the 
non-trivial action is always given by the integral over $c_i$.

\begin{remark}  \label{rem:hyperelliptic}
The actions $I_i$ are hyperelliptic integrals of the third kind, with
poles at $z = b_{\sigma}$ (recall that we assumed $m_{\sigma} \geq 2$ for all $\sigma$).
In the non-degenerate Neumann system the actions are given 
by hyperelliptic integrals of the second kind~\cite{dullin01}.
Consider as the integration path a circle $\tilde c_\sigma$ 
around the pole $z = b_\sigma$. Then 
\begin{displaymath}
   J_\sigma = \frac{1}{2\pi}
   \oint_{\tilde c_\sigma} \frac{\zeta}{A(z)} \, \dee z \,.
\end{displaymath}
In this sense the hyperelliptic curve ``knows'' not only about $I_i$ but
also about the trivial actions~$J_\sigma$. 
The reason for this is the identity
\begin{displaymath}
   \zeta^2(b_{\sigma}) = \tilde{Q}(b, w ; b_{\sigma})
   = w_{\sigma} \prod_{\tau \neq \sigma} (b_{\sigma} - b_{\tau})^2
\end{displaymath}
so that the residue of the action integrand at $z = b_{\sigma}$ is 
the action $ \sqrt{- w_{\sigma}}$.
This is a common phenomenon: if there are simple actions (e.g.\ actions
that are generators of a global symmetry) then the non-trivial actions
of the system are abelian integrals of the third kind and the integration
around the poles gives the trivial actions or multiples thereof.
\end{remark}

\noindent
The integrals $I_1, \ldots, I_{\ell}$ depend on the separation
constants~$\rho_i$ and on the trivial actions~$J_{\sigma}$.
Next to the values $j = (j_0, \dots, j_\ell)$ of these we also
write $h = (h_1, \dots, h_\ell)$ where $h_1$ is the value of the energy
and $h_i$ are the values of the other smooth constants of motion $H_i$
not related to the symmetry group~$G$.
These may e.g.\ be the values of $\ell - 1$ of the $\ell + 1$
integrals~\eqref{eqn:integralsdegenerate}, or the separation constants
$\rho_2, \ldots, \rho_{\ell}$.
The action mapping ${\cal A}$ then takes a regular value
$(h, j) \in \R^\ell \times \R^{\tilde\ell + 1}$
and assigns the actions ${\cal A} = (I, J)$.
The period lattice and the frequencies are determined by the derivatives
of~${\cal A}$.
The vector fields of the actions (written as a vector) can be
expressed in terms of the vector fields of the $2 \ell + 1$ functions
$(H, J)$ as
\begin{displaymath}
   \begin{pmatrix} \X{I} \\ \X{J} \end{pmatrix} =
   T \begin{pmatrix} \X{H} \\ \X{J} \end{pmatrix}
   \quad \mbox{with} \quad  T =
   \begin{pmatrix}
      \frac{\partial I}{\partial H} & \frac{\partial I}{\partial J} \\
      \frac{\partial J}{\partial H} & \frac{\partial J}{\partial J}
   \end{pmatrix} \,.
\end{displaymath}
The matrix $T$ describes the period lattice of the integrable system.
The frequencies of the Hamiltonian flows of $\X{H_i}$ and~$\X{J_{\sigma}}$
are denoted by $\Omega_i$, $\Omega_{\ell + \sigma + 1}$ and taken as
row vectors they form the frequency matrix~$\Omega$. 
By definition the frequencies of a constant of motion are the coefficients
of the decomposition of its vector field in terms of the vector fields of
the actions, hence
\begin{displaymath}
   \begin{pmatrix} \X{H} \\ \X{J} \end{pmatrix} =
   \Omega
   \begin{pmatrix} \X{I} \\ \X{J} \end{pmatrix}
\end{displaymath}
and therefore $\Omega = T^{-1}$ and in particular 
\begin{displaymath}
   \X{H_1} = \Omega_1 \begin{pmatrix} \X{I} \\ \X{J} \end{pmatrix} \,,
\end{displaymath}
so that the frequencies~$\Omega_1$ of the Hamiltonian are the entries in 
the first row of $\Omega = T^{-1}$, hence $\Omega_1 = T^{-t} \e_1$.

In the present case the matrix $T$ has a simple block structure caused by 
 the trivial actions $J_{\sigma}$, namely
\begin{displaymath}
   T =
   \begin{pmatrix}
      \frac{\partial I}{\partial H} & \frac{\partial I}{\partial J} \\
      0 & \id_{\ell+1}
   \end{pmatrix} \,.
\end{displaymath}
Therefore, the frequency matrix is given by 
\begin{displaymath}
   \Omega =
   \begin{pmatrix}
      \left( \frac{\partial I}{\partial H} \right)^{-1} &
      -\left( \frac{\partial I}{\partial H} \right)^{-1}
      \frac{\partial I}{\partial J} \\
      0 & \id_{\ell+1}
   \end{pmatrix} \,.
\end{displaymath}
The $(\ell + 1)$--dimensional identity matrix expresses the fact that $J$
consists of global smooth actions, and hence each of their flows has
frequency one.

A particularly simple case occurs when the reduced system has one degree of 
freedom, $\ell = 1$.
Then the upper left block is a scalar which is also the first entry in the 
frequency vector $\Omega_1$.
Frequency ratios are formed by dividing $\Omega_1$ by one of its entries,
so the scalar $\partial H/\partial I$ cancels and using the first entry
the two frequency ratios are simply $\frac{\partial I}{\partial J_0}$ and
$\frac{\partial I}{\partial J_1}$.

\section{Global structure of the flow}
\label{globalstructureoftheflow}

Let us first comment on some special properties that appear 
when the multiplicity $m_\sigma$ of roots of the quadratic potential in the 
Neumann system is small.

\begin{description}
   \item{$m_{\sigma} = 1$.}
In this case the eigenvalue $a_{\nu} = b_{\sigma}$ is not equal to
any of the other eigenvalues of the quadratic potential.
The factor $O(1)$ of the symmetry group~$G$ is discrete and does
not yield a component of the Momentum mapping~$\J$ --- we formally
set $J_{\sigma} \equiv 0$.
Furthermore it may be preferable not to reduce with respect to this
discrete factor, as this only leads to singular points without
the reward of lowering the number of degrees of freedom.
See also Remarks~\ref{rem:SingRed} and~\ref{rem:Rosochatius}.
   \item{$m_{\sigma} = 2$.}
In this case reduction lowers the number of degrees of freedom by one, but
$2$ equal eigenvalues of the potential do not yet contribute
to a possible superintegrability of the degenerate Neumann system.
In particular, this means that we may observe monodromy by only
reducing the $SO(2)$--factor and not the whole $O(2)$--factor.
This keeps the sign in the value of the Casimir which we rather
choose to be $\J_{\sigma} = L_{ik}$ than $J_{\sigma} = \sqrt{L_{ik}^2}$.
In this way it becomes possible to find a closed non-contractible
loop around a critical value of the Energy--Momentum mapping, which is
located on the boundary of the image of the Energy--Casimir mapping.
See also Remarks~\ref{rem:polynomial}, \ref{rem:signed}, \ref{rem:Rosochatius},
and~\ref{rem:nice} below.
   \item{$m_{\sigma} = 3$.}
As soon as we have at least one group of three (or more) equal
eigenvalues~$a_{\nu}$ of the potential, the degenerate Neumann
system becomes superintegrable.
Still, as long as there are no groups of four or more equal eigenvalues
the weakly regular values of the Momentum mapping $\J$ are in fact
regular values --- there are no syzygies constraining the~$L_{ik}$.
   \item{$m_{\sigma} = 4$.}
This case is still distinguished from the general
$m_{\sigma} \geq 5$ in that the Gra{\ss}mannian
$G_{4,2} \subseteq J_{\sigma}^{-1}(|| \mu_{\sigma} ||) \subseteq \oo(4)$
is a hypersurface of~$J_{\sigma}^{-1}(|| \mu_{\sigma} ||)$, defined by
the zero set of the second global Casimir
$L_{12} L_{34} + L_{14}L_{23} + L_{13} L_{24}$ (next to
$W_{\sigma} = L_{12}^2 + L_{13}^2 + L_{14}^2 + L_{23}^2 + L_{24}^2 + L_{34}^2$)
of~$\oo(4)$.
\end{description}

\subsection{Integral--Casimir Mapping}
\label{image}

In order to describe another aspect of the reduced degenerate Neumann
system we introduce yet another mapping $\IC$ in addition to $\EM$ and~$\EC$ 
considered in Section~\ref{reduce}.
The  name Integral--Casimir mapping captures the fact that there are two 
different kinds of conserved quantities in the system; those that come
from the $G$--action and those that do not.
The momenta of the $G$--action yield Casimirs for $m_{\sigma} \geq 2$.
The other integrals (and of course the energy) are simply
conserved quantities that are not related to the symmetry group~$G$.

Regular reduction together with separation of variables describes 
the system when all $J_{\sigma} > 0$ and when the motion stays outside the 
singularities of the separating co-ordinate system.
Singularities of the separating co-ordinate system occur when one of the roots~$u_i$
is equal to $b_j$, or when $u_i = u_{i+1}$ for some $i$. 
Hence, whenever the hyperelliptic curve has no double roots 
the motion on the corresponding torus is well described by the
elliptical-spherical co-ordinates.
To find the candidates where the Integral--Casimir mapping $\IC$ does not have
full rank we can therefore simply study the discriminant locus of the
hyperelliptic curve. 

\begin{lemma} \label{lem:discriminantlocus}
   The discrimiant locus for the reduced Neumann system with all
   $w_{\sigma} \neq 0$ is topologically a $2^{\ell}$--tant without
   any internal structure.
\end{lemma}

\begin{proof}
The position of the roots of the hyperelliptic curve is such that no
root can be located at $b_{\sigma}$ because setting $z = b_{\sigma}$
in~$\tilde{Q}$ gives
\begin{displaymath}
   \tilde{Q}(b_{\sigma}) = -w_{\sigma} A'(b_{\sigma})^2 \not  = 0\,,
\end{displaymath}
which implies that all the roots $z=b_{\sigma}$ for $w_{\sigma} = 0$ 
are not roots any more when $w_{\sigma} \neq 0$.
For a real motion there must be a positive interval for each $u_i$
between $b_{i-1}$ and $b_{i}$.
Hence there must be two roots in each such interval.
This uses up $2\ell$ roots already.
The remaining root must be to the left of $b_0$, because
$\zeta^2 \to +\infty$ for $z \to - \infty$.
Now all pairs of roots bounding the $u_i$ intervals are separated
by some $b_{\sigma}$, and hence for $w_{\sigma} \neq 0$ the only
possible collision is within a pair, leading to $u_i = u_{i+1}$.
Making $k$ such pairs collide gives a corank $k$ degenerate value in 
the image of the Energy--Momentum mapping.
The corank $k$ critical values give the $(\ell - k)$--dimensional faces
of the $2^{\ell}$--tant.
Making all pairs collide gives the corank $\ell$ elliptic (relative)
equilibrium point of the reduced system.
This corresponds to the origin of the $2^{\ell}$--tant.
\end{proof}

\noindent
For $\ell = 2$ the discriminant locus of the curve of the reduced system
is explicitly given by
\begin{displaymath}
\begin{aligned}
   \rho_1 = & -s + \sum_{\sigma = 0}^2
   w^2_{\sigma} \frac{A'(b_{\sigma})}{2(s-b_{\sigma})^2} \\
   \rho_2 = & s^2 + \sum_{\sigma = 0}^2
   w^2_{\sigma} \frac{A'(b_{\sigma})}{(s-b_{\sigma})}
   \left( 1 + \frac{b_{\sigma}}{2(s-b_{\sigma})} \right) \,.
\end{aligned}
\end{displaymath}
When $w_{\sigma} = 0$ then in addition the straight line given by
$Q(b_{\sigma}) = 0$ needs to be added to the discriminant locus.

\subsection{Convexity}
\label{convexity}

The simple discrimiant loci described in the previous
Lemma~\ref{lem:discriminantlocus} are the fibres over $J$--space.
In this view the $J$--space is the parameter space for the reduced system. 
Instead of the discriminant locus of the reduced system we now study the
$J$--space as the image of the Casimir mapping of the full system.
More precisely, we consider the image of the Momentum mapping of the
$\T^{\ell+1}$--action on phase space given by $J$,
for short the Casimir mapping.
To obtain a compact image we fix the energy, and study the set of possible
values $j$ of~$J$.

For each $j$ there is a unique elliptic (relative) equilibrium of the
reduced system. 
This corresponds to the origin of the $2^{\ell}$--tant of
Lemma~\ref{lem:discriminantlocus}, and there is a corresponding energy. 
Fixing instead the energy and asking for all equilibria with any value~$j$
gives the boundary of the image of the Casimir mapping restricted to
constant energy.

\begin{theorem}
   The image of the Energy--Casimir mapping is a convex set for sufficiently
   large~$h$.
\end{theorem} 

\begin{proof}
It is important to use the actions $J_{\sigma}$ as co-ordinates in the image,
not their squares.
The result follows from the explicit parametrisation of the discriminant
locus.
We equate $ - QA + \tilde{Q}$ to $-(z-s)^2 (z - r_1) \dots (z - r_{2\ell - 1})$
and solve for $j_{\sigma}$ and~$\rho_i$.
This leads to
\begin{displaymath}
   \rho_1 = -s + \frac{1}{2}\sum_{\sigma = 0}^{\ell} b_{\sigma}
   - \frac{1}{2} \sum_{k=1}^{2 \ell - 1} r_k
\end{displaymath}
and 
\begin{displaymath}
    j_{\sigma} = \frac{b_{\sigma} - s}{A'(b_{\sigma})} 
     \left ( \prod_{k = 1}^{2\ell-1} (b_{\sigma} - r_k) \right)^{1/2} \,.
\end{displaymath}
Consider now a parametrisation of the corank $\ell$ stratum of the
discriminant locus $\Sigma$, which gives the relative equilibria which
mark the boundary of the image of the Casimir mapping.
Hence we parametrise by the~$\ell$ double roots, $s_k$, $k = 1, \dots, \ell$
and find as a special case of the above formula the values where $r_k$
are made pairwise equal and $r$~denotes the remaining single root
\begin{displaymath}
   j_{\sigma} = 
   \frac{ \sqrt{ b_{\sigma} - r } }{A'(b_{\sigma})}  
   \prod_{k=1}^{\ell} (b_{\sigma} - s_k)
\end{displaymath}
and at the same time $\rho = \rho(s_k, r)$.
Since $\rho_1$ is linear in $r$ we can eliminate $r$ in favour of
the value $h = \rho_1$ of the Hamiltonian and the sum of the~$s_k$. 
Also we introduce the symmetric functions $t_k$ of the $s_k$ and find
\begin{displaymath}
   j_{\sigma} = \frac{\omega_\sigma}{A'(b_{\sigma})} 
   ( b_{\sigma}^{\ell} + b_{\sigma}^{\ell-1} t_1 + \dots + t_{\ell}), 
   \quad \mbox{with} \quad \omega_\sigma^2 = h + b_\sigma - 2 t_1 \,.
\end{displaymath}
The parametrisation is linear in $t_k$ for all $k > 1$.

For sufficiently large energy the boundary of the image of the
Energy--Casimir mapping is given by the critical values
$j_\sigma(h, t_1, \dots, t_\ell)$.
This implicitly defines  $h_c(j_0, \dots, j_\ell)$, the value of the
critical energy as a function of the momenta.
By implicit differentiation we compute the gradient of the energy
\begin{displaymath}
  \frac{ \partial h_c}{\partial j_\sigma} = 2 \omega_\sigma
\end{displaymath}
and the Hessian as
\begin{displaymath}
   \frac{ \partial^2 h_c}{\partial j_\sigma j_\tau}
   = 2 \frac{\polw}{\polpo} \frac{1}{\omega_\sigma \omega_\tau}, \quad
   \polpo = \sum_{\tau = 0}^{\ell} (h-2t_1)^{\ell-\tau} (-1)^\tau t_\tau , 
   \quad \polw = \prod_{\tau = 0}^{\ell} \omega_\tau^2
\end{displaymath}
This is a rank 1 symmetric matrix with nonzero eigenvalue
\begin{displaymath}
   \lambda = 2 \frac{\polw}{\polpo}
       \sum_{\sigma = 0}^{\ell} \frac{1}{ \omega_\sigma^2 } 
\end{displaymath}
whose eigenvector is $(1/\omega_0, \dots, 1/\omega_\ell)$.
As a result the Taylor expansion of the critical energy is
\begin{displaymath}
   h_c(j + \Delta) = 2 \sum_{\sigma = 0}^{\ell} \omega_\sigma \Delta_\sigma + 
   2 \frac{\polw}{\polpo}
   \left( \sum_{\sigma = 0}^{\ell} \frac{\Delta_\sigma}{\omega_\sigma}\right)^2
   + O(\Delta^3) \,.
\end{displaymath}
If $\polpo > 0$ then the Hessian is positive semi-definite, and the boundary
of the image of the Energy--Casimir mapping is the graph of a convex function,
and hence the image is a convex set.
Now $b_i < s_{i+1} < b_{i+1}$, and with $t_1 = - \sum s_i$ we find
\begin{displaymath}
   B - b_0 < -t_1 < B - b_\ell , \quad
   B = \sum_{\sigma = 0}^{\ell} b_\sigma
\end{displaymath}
Since all $\omega_\sigma$ must be real, we conclude that the image is convex if 
the energy is sufficiently large, namely when $h > - 2(B-b_\ell) - b_0$.
Rewriting $t_i$ in the roots $s_i$ gives
\begin{displaymath}
   \polpo =  \prod_{i=1}^{\ell} ( h + s_i + 2S), \quad
   S =  \sum_{i=1}^{\ell} s_i \,.
\end{displaymath}
From the above $h + 2 S > -2(B-b_\ell) - b_0 + 2B - 2b_0 = 2b_\ell - 3 b_0$
and therefore $h + s_k + 2S > 2 b_\ell - 2 b_0 > 0$ and $P > 0$ as desired.
\end{proof}

\begin{theorem}
For sufficiently large fixed energy the image of the Casimir mapping~$J$
given by the $\T^{\ell+1}$--action is convex. 
For $h \to \infty$ the boundary tends to a convex polyhedron.
\end{theorem}

\begin{proof}
Since we have shown that $h_c(j_0, \dots, j_\sigma)$ is a convex function, and
$h_c$ is the boundary of the image of the Energy--Casimir mapping for
$h > - 2(B-b_\ell) - b_0$, every slice of constant energy through this convex
set is again a convex set.
This set is the image of the Casimir mapping applied to the surface of constant
energy.
In the limit $h \to \infty$ it becomes a convex polyhedron when rescaled
by $\sqrt{h}$.
This amounts to setting $\omega_\sigma = 1$ in the limit, and so the
non-linearity disappears. 
Even for finite $h$ the boundary has zero curvature for $\ell \ge 2$.
In fact the boundary is a ruled surface for $\ell = 2$ and a higher dimensional
analogue with an $\ell-1$ dimensional plane attached at each point of
the hypersurface.
\end{proof}

\begin{remark}  \label{rem:nice}
The above statement about convexity becomes particularly nice in the case that all $m_\sigma = 2$.
Then the $J_\sigma$ can all be chosen to be signed quantities, and are defined
also for~$J_\sigma = 0$, so that there is a global $\T^{\ell+1}$--action
on phase space.
In this way the convex set extends through the faces $J_\sigma = 0$.
Note that this convexity result is obtained even though 
the phase space is not compact.
The compactness is achieved by restricting to the compact energy
surfaces $h = \const$.
\end{remark}

\appendix
\section*{Acknowledgement}

We like to thank Chris Davison for carefully reading the manuscript.
Furthermore we thank the Mathematics Institute at the University of Warwick
(where this work was begun) and the Centre Interfacultaire Bernoulli at the
EPF Lausanne (where this work was nearly finished) for their hospitality.
This research was partially supported by the European Research Training Network
{\it Mechanics and Symmetry in Europe\/} (MASIE), HPRN-CT-2000-00113.
HRD was supported in part by ARC grant DP110102001.

\section{Jacobi Trick}
\label{jacobitrick}

Let 
\[
   U(z) = \prod_{i=1}^{\ell} (z-u_i)
\]
be a polynomial of degree $\ell$ with {\em distinct} roots $u_i$.
Then the derivative of U evaluated at the root $u_k$ is
\[
   U'(u_k) = \prod_{i=1, i \neq k}^{\ell} (u_k - u_i) \,.
\]
This gives a way to exclude one factor in $U$ and thus the partial fraction decomposition of $1/U$ is
\[
   \frac{1}{U(z)} = \sum_{i=1}^{\ell} \frac{1}{U'(u_i)( z - u_i)} \,.
\]
The identity (the use of which is sometimes called Jacobi's trick)
\[
   \sum_{i=1}^{\ell} u_i = \sum_{i=1}^{\ell} \frac{u_i^{\ell} }{U'(u_i) } 
\]
follows from calculation of 
\[
   \frac{1}{2\pi} \oint \frac{z^{\ell}}{U(z)} \dee z
\] 
along a countour that is a sufficiently large
circle that encloses all poles $u_i$.
The left hand side is found by calculation of 
the residue at infinity, 
while the right hands side is found from the sum over the finite residues.
In a similar way the identity 
\[
   1 = \sum_{i=1}^{\ell} \frac{u_i^{\ell - 1}}{U'(u_i)} 
\]
can be obtained.
Lowering the exponent further the residue at infinity vanishes.
Jacobi's trick is used in the separation of variables
in order to introduce the separation constants.
For powers of $z$  greater than $\ell$ higher order symmetric functions
of the roots are found.

\begin{footnotesize}

\end{footnotesize}

\end{document}